\documentclass[11pt, twoside, a4paper]{article}
\usepackage{amsmath,amssymb,amsthm,amscd,appendix,color,epsfig,eucal,enumitem,dsfont,graphics,graphicx,latexsym,verbatim, hyperref}
\newtheorem{theorem}{Theorem}
\newtheorem{corollary}[theorem]{Corollary}
\newtheorem{lemma}[theorem]{Lemma}
\newtheorem{definition}[theorem]{Definition}

\newtheorem{notation}[theorem]{Notation}
\newtheorem{proposition}[theorem]{Proposition}
\newtheorem{example}[theorem]{Example}

\newcommand{\keywords}[1]{\par\addvspace\baselineskip\noindent\textbf{Keywords:}\enspace\ignorespaces#1}
\newcommand{\AMSclassification}[1]{\par\addvspace\baselineskip\noindent\textbf{Mathematical subject classification:}\enspace\ignorespaces#1}

\begin{document}
\title{Calibrated configurations for Frenkel-Kontorova type models in almost-periodic environments}
\author{
\small{Eduardo Garibaldi}\thanks{supported by FAPESP 2009/17075-8, CAPES-COFECUB 661/10 and Brazilian-French Network in Mathematics}\\
\footnotesize{Departamento de Matem\'atica}\\
\footnotesize{Universidade Estadual de Campinas}\\
\footnotesize{13083-859 Campinas - SP, Brasil}\\
\footnotesize{\texttt{garibaldi@ime.unicamp.br}}
\and
\small{Samuel Petite}\thanks{supported by FAPESP 2009/17075-8, CAPES-COFECUB 661/10 and ANR SUBTILE 0879. }\\
\footnotesize{LAMFA,  CNRS, UMR 7352 }\\
\footnotesize{Universit\'e de Picardie Jules Verne}\\
\footnotesize{80000 Amiens, France}\\
\footnotesize{\texttt{samuel.petite@u-picardie.fr}}
\and
\small{Philippe Thieullen}\thanks{supported by FAPESP 2009/17075-8}\\
\footnotesize{Institut de Math\'ematiques, CNRS, UMR 5251}\\
\footnotesize{Universit\'e Bordeaux 1}\\
\footnotesize{F-33405 Talence, France}\\
\footnotesize{\texttt{Philippe.Thieullen@math.u-bordeaux1.fr}}
}
\date{\today}
\maketitle

\begin{abstract}
The Frenkel-Kontorova model describes how an infinite chain of atoms minimizes the total energy of the system  when the energy takes into account the interaction of nearest neighbors as well as the interaction with an exterior environment. An almost-periodic environment leads to consider a family of interaction energies which is stationary with respect to a minimal topological dynamical system. We introduce, in this context, the notion of calibrated configuration (stronger than the standard minimizing  condition) and, for continuous superlinear interaction energies, we show the existence of these configurations for some environment of the dynamical system.  Furthermore, in one dimension, we give sufficient conditions on the family of interaction energies to ensure, for any environment, the existence of calibrated configurations when the underlying dynamics is uniquely ergodic. The main mathematical tools for this study are developed in the frameworks of discrete weak KAM 
theory, Aubry-Mather theory and  spaces of Delone sets.


\AMSclassification{37B50, 37J50, 37N20, 49L20, 49L25, 52C23}
\end{abstract}

\section{Introduction}

A {\it minimizing configuration} $\{x_k\}_{k\in\mathbb{Z}}$ for an {\it interaction energy} $E:\mathbb{R}^d \times \mathbb{R}^d \to \mathbb{R}$ is a chain of points in  $\mathbb{R}^d$ arranged so that the energy of each finite segment $(x_m,x_{m+1},\ldots,x_n)$ cannot be lowered by changing the configuration inside the segment  while fixing the two boundary points. Define
\begin{gather*}
E(x_m,x_{m+1},\ldots,x_n) := \sum_{k=m}^{n-1} E(x_k,x_{k+1}).
\end{gather*}
Then $\{x_k\}_{k\in\mathbb{Z}}$ is said to be minimizing if, for all $ m<n $, for all $ y_m,y_{m+1},\ldots,y_n \in \mathbb{R}^d $ satisfying $y_m=x_m$ and $y_n=x_n$, one has
\begin{equation}\label{equation:minimizing}
 E(x_m,x_{m+1},\ldots,x_n) \leq E(y_m,y_{m+1},\ldots,y_n).
\end{equation}
If the interaction energy is $C^0$, {\it coercive} and {\it translation periodic},
\begin{gather}
\lim_{R\to+\infty}\inf_{\|y-x\|\geq R} E(x,y) = +\infty, \label{equation:DefinitionCoercive} \\
\forall \, t \in \mathbb{Z}^d, \ \forall \, x,y\in\mathbb{R}^d, \quad E(x+t,y+t)=E(x,y), \label{equation:DefinitionPeriodic}
\end{gather}
it is easy to show (see \cite{GaribaldiThieullen2011_01}) that minimizing configurations do exist. If $d=1$ and $E$ is a smooth
{\it strongly twist} translation periodic interaction energy,
\begin{equation}
\label{equation:StronglyTwistProperty}
\frac{\partial^2 E}{\partial x\partial y} \leq -\alpha <0,
\end{equation}
a minimizing configuration admits in addition a rotation number (see Aubry and Le Daeron \cite{AubryLeDaeron1983_01}). The interaction energy $E$ is supposed to model the interaction between two successive points as well as the interaction between the chain and the environment.

For environments which are aperiodic, namely, with trivial translation group, few results are known (see, for instance, \cite{DelaLlaveSu, GambaudoGuiraudPetite2006_01, vanErp}). If $d=1$ and $E$ is a twist interaction energy describing a quasicrystal environment, Gambaudo, Guiraud and Petite \cite{GambaudoGuiraudPetite2006_01} showed that minimizing configurations do exist, they all have a rotation number and any prescribed real number is the rotation number of a minimizing configuration.

We shall make slightly more general assumptions on the properties of $E$.
We say that $E$ is {\it translation bounded} if
\begin{equation}
\label{equation:TranslationBounded}
\forall \, R>0, \quad \sup_{\|y-x\| \leq R} E(x,y) < +\infty,
\end{equation}
\noindent  {\it translation uniformly continuous} if
\begin{equation}
\label{equation:TranslationContinuous}
\forall \, R>0,  \quad\text{$E(x,y)$ is uniformly continuous in $\|y-x\| \leq R$},
\end{equation}
\noindent and {\it superlinear} if
\begin{equation}
\label{equation:SuperLinear}
\lim_{R\to+\infty} \ \inf_{\|y-x\| \geq R} \ \frac{E(x,y)}{\|y-x\|} =+\infty.
\end{equation}

In a parallel work (see \cite{GPT}), we discuss the existence of semi-infinite minimizing configurations in any dimension for a superlinear, translation bounded and translation uniformly continuous $E$. It is not clear that there exist bi-infinite minimizing configurations in this general context.

We call {\it ground energy} the lowest energy per site for all configurations
\begin{equation}
\bar E := \lim_{n\to+\infty} \ \inf_{x_0,\ldots,x_n} \frac{1}{n} E(x_0,\ldots,x_n).
\end{equation}
A configuration $\{x_n\}_{n\in\mathbb{Z}}$ is {\it calibrated} at the level $\bar E$ if, for every $k<l$,
\begin{gather}
\label{equation:CalibratedConfiguration}
E(x_k,\ldots,x_l) - (l-k)\bar E \leq \inf_{n \geq1} \ \inf_{y_0=x_k,\ldots,y_n=x_l} \big[E(y_0,\ldots,y_n) -n\bar E\big].
\end{gather}
Notice that the number of sites on the right hand side is arbitrary. A calibrated configuration is obviously minimizing; the converse is false in general. More generally, a configuration which is calibrated at some level $c$ (replace $\bar E$  by $c$ in (\ref{equation:CalibratedConfiguration})) is also minimizing.

If $d\geq1$ and $E$ is $C^0$, coercive and translation periodic (conditions (\ref{equation:DefinitionCoercive}) and~(\ref{equation:DefinitionPeriodic})), an argument using the notion of weak KAM solutions as in \cite{Gomes2005_01,Fathi2010_01,GaribaldiThieullen2011_01} shows that there exist calibrated configurations at the level $\bar E$. Conversely, if $d=1$ and $E$ is twist translation periodic, every minimizing configuration is calibrated for some modified energy $E_\lambda(x,y)=E(x,y)-\lambda(y-x)$, $\lambda\in\mathbb{R}$, with ground energy $\bar E(\lambda)$. If $d=1$ and $E$ is arbitrary (at least translation bounded, translation uniformly continuous and superlinear but not translation periodic), it is not known in general that a calibrated configuration does exist.

In order to give a positive answer to the question of the existence of calibrated configurations, we will consider in this paper an
interaction energy which has  almost periodic behavior. This leads to look at a family of interaction energies parameterized by a minimal topological dynamical system. Such an approach is similar  to studies for the Hamilton-Jacobi equation (see, for instance, \cite{DaviniSiconolfi2009,DaviniSiconolfi2011,  DaviniSiconolfi2012, LionsSouganidis}), where a stationary ergodic setting has been taken 
into account.

Concretely, we  will  assume there exists a family of interaction energies $\{E_\omega\}_{\omega}$ depending on an environment $\omega$. Let    $\Omega$ denote the collection of all possible  environments.  We assume that a chain $\{x_k+t\}_{k\in\mathbb{Z}}$ translated in the direction $t\in\mathbb{R}^d$ and interacting with the environment $\omega$ has the same local energy that $\{x_k\}_{k\in\mathbb{Z}}$ interacting with the shifted environment $\tau_t(\omega)$, where $\{\tau_t:\Omega\to \Omega\}_{t\in\mathbb{R}^d}$ is supposed to be a group of bijective maps. More precisely, each environment $\omega$ defines an interaction $E_\omega(x,y)$ which is assumed to be {\it topologically stationary} in the following sense
\begin{equation}
\label{equation:Equivariance}
\forall \, \omega\in \Omega, \ \forall \, t\in \mathbb{R}^d, \ \forall \, x,y \in \mathbb{R}^d, \quad E_\omega(x+t,y+t) = E_{\tau_t(\omega)}(x,y).
\end{equation}

In order to ensure the topological stationarity, the interaction energy will be supposed to have a {\it Lagrangian form}.
Formally, we will use the following definition.

\begin{definition}\label{notacao}
Let $\Omega$ be a compact metric space.
\begin{enumerate}
\item An almost periodic environment is a couple $\big(\Omega,\{\tau_t\}_{t\in\mathbb{R}^d}\big)$, 
where  $\{\tau_t\}_{t\in\mathbb{R}^d}$ is  a minimal $\mathbb{R}^d$-action, that is, a family of 
homeomorphisms $\tau_t:\Omega\to\Omega$ satisfying 
\par-- $ \tau_s\circ \tau_t = \tau_{s+t}$ for all $s,t \in \mathbb{R}^d$  (the cocycle property),
\par-- $\tau_t(\omega)$ is jointly continuous with respect to $(t,\omega)$,
\par-- $\forall \, \omega \in\Omega, \ \{\tau_t(\omega)\}_{t\in\mathbb{R}^d}$ is dense in $\Omega$.

\item A family of interaction energies $\{E_\omega\}_{\omega\in\Omega}$ is said to derive from a Lagrangian if there exists a continuous  function  $L:\Omega\times\mathbb{R}^d \to \mathbb{R}$ such that
\begin{equation}\label{Lagrangianform}
\forall \, \omega \in\Omega, \ \forall \, x,y \in \mathbb{R}^d, \quad E_\omega(x,y):=L(\tau_x(\omega),y-x).
\end{equation}
\item An almost periodic interaction model is the set of data $(\Omega,\{\tau_t\}_{t\in\mathbb{R}^d},L)$  where $(\Omega,\{\tau_t\}_{t\in\mathbb{R}^d})$ is an almost period environment and $L$ is a continuous function on $\Omega \times \mathbb{R}^d$.
\end{enumerate}
\end{definition}

Notice  that  the expression ``almost periodic" shall not be understood in the sense of H.  Bohr.  The almost periodicity according to 
Bohr is canonically relied to the uniform convergence. See \cite{AuslanderHahn} for a discussion on the different concepts of almost periodicity in conformity with the uniform topology or with the compact open topology.

Because of the particular form~\eqref{Lagrangianform} of  $E_\omega(x,y)$, these energies are translation bounded and translation continuous uniformly in $\omega$ and in $\|y-x\| \leq R$. We make precise the two notions of {\it coerciveness} and {\it superlinearity} in the Lagrangian form.

\begin{definition}
Let $(\Omega,\{\tau_t\}_{t\in\mathbb{R}^d},L)$ be an almost periodic interaction model. The Lagrangian $L$ is said to be coercive  if
\[
\lim_{R\to+\infty} \ \inf_{\omega\in \Omega} \ \inf_{\|t\|\geq R} L(\omega,t) = +\infty.
\]
$L$ is said to be superlinear if
\[
\lim_{R\to+\infty} \ \inf_{\omega\in \Omega} \ \inf_{\|t\|\geq R} \frac{L(\omega,t)}{\|t\|} = +\infty.
\]
$L$ is said to be ferromagnetic if, for every $\omega \in \Omega$, $E_\omega$ is of class $C^1(\mathbb{R}^d\times \mathbb{R}^d)$ and, for every $\omega \in \Omega$ and $x,y \in \mathbb{R}^d$,
\[
x\in\mathbb{R}^d \mapsto \frac{\partial E_\omega}{\partial y}(x,y) \in \mathbb{R}^d \quad \text{and} \quad y\in\mathbb{R}^d \mapsto \frac{\partial E_\omega}{\partial x}(x,y) \in \mathbb{R}^d
\]
are homeomorphisms.
\end{definition}

Note that if there is a constant $\alpha>0$ such that $\sum_{i,j=1}^d \frac{\partial^2 E_\omega}{\partial x\partial y}v_i v_j \leq -\alpha \sum_{i=1}^d v_i^2$ for all $\omega\in \Omega, x,y \in\mathbb{R}^d$, then $L$ is ferromagnetic and superlinear.

Let us illustrate our abstract notions by three typical examples.
\begin{example}\label{ex:1}
The classical periodic one-dimensional Frenkel-Kontorova model \cite{FrenkelKontorova1938} 
takes into account  the family of interaction energies
$E_\omega(x,y) = W(y-x) + V_\omega(x)$, with $ \omega \in  \mathbb{S}^1 $, written in Lagrangian form as
\begin{equation}
L(\omega,t) = W(t)+V(\omega) =  \frac{1}{2}|t-\lambda|^2 + \frac{K}{(2\pi)^2}\big( 1-\cos 2\pi \omega \big),
\end{equation}
\noindent where $\lambda$, $K$ are constants. Here $\Omega = \mathbb{S}^1$ and $\tau_t : \mathbb{S}^1 \to \mathbb{S}^1$ is given by  $\tau_t(\omega)=\omega+t$. We observe that $\{\tau_t\}_t$  is clearly minimal.
\end{example}

The following example comes from  \cite{GambaudoGuiraudPetite2006_01}.
\begin{example}\label{ex:3} Consider, for an irrational $\alpha \in (0,1)\setminus \mathbb{Q}$,  the set
$$ \omega(\alpha):= \{n \in \mathbb Z : \lfloor n \alpha \rfloor - \lfloor (n-1) \alpha \rfloor  = 1 \},$$
where $\lfloor \cdot \rfloor$ denotes the integer part. Notice that the  distance between two consecutive elements of $\omega(\alpha)$ is $\lfloor \frac 1 \alpha \rfloor$ or $\lfloor \frac 1 \alpha \rfloor +1$. Now let $U_{0}$ and $U_{1} $ be two real valued  smooth functions with  supports  respectively in $(0,\lfloor \frac 1 \alpha \rfloor)$ and $(0,\lfloor \frac 1 \alpha \rfloor+1)$. Let $V_{\omega(\alpha)}$ be the function defined by $V_{\omega(\alpha)}(x) = U_{\omega_{n+1} - \omega_{n} - \lfloor \frac 1 \alpha \rfloor}(x-\omega_{n})$, where $\omega_{n} < \omega_{n+1}$ are the two consecutive elements of the set $\omega(\alpha)$ such that $\omega_{n} \leq x < \omega_{n+1}$. The associated interaction energy is the function
\begin{equation}\label{equation:FKSturm}
E_{\omega(\alpha)}(x,y) = \frac 12 |x-y-\lambda|^2 + V_{\omega(\alpha)}(x).
\end{equation}
Given any relatively dense set $\omega'$ of the real line such that the distance between two consecutive points is in 
$\{\lfloor \frac 1 \alpha \rfloor, \lfloor \frac 1 \alpha \rfloor +1\}$, we can directly extend the previous definition and introduce the function $V_{{\omega'}}$. Let $\Omega'$ be the collection of all such sets~$\omega'$. Then, for any $x,t \in \mathbb R$, we have the relation $V_{\omega'}(x+t) = V_{\omega'-t}(x)$, where $\omega' -t$ denotes the set of elements of $\omega' \in \Omega'$ translated by $-t$.  In section \ref{section:Frenkel-Kontorova-quasicrystals}, we explain how to associate a compact metric space $\Omega \subset \Omega' $, where the group of translations acts minimally, as well as a Lagrangian from which the family $\{E_{\omega}\}_{\omega\in \Omega}$ derives.
\end{example}

As we shall see in section \ref{section:Frenkel-Kontorova-quasicrystals}, the construction  given in  example \ref{ex:3} extends to any \emph{quasiscrystal} $\omega$ of $\mathbb R$, namely,  to   any set  $\omega \subset \mathbb R$ which is relatively dense and uniformly discrete such that  the difference set $\omega-\omega$ is discrete and any finite pattern repeats with a positive frequency. We will later focus 
on a particular class of interaction models, called almost crystalline, which will include all quasicrystals.
An example of almost periodic interaction model on $\mathbb R$ which is not  almost crystalline  can be constructed in the following way.

\begin{example}\label{ex:2}
The underlying minimal flow is the irrational flow $\tau_t(\omega)=\omega+t(1,\sqrt{2})$ acting on $\Omega=\mathbb{T}^2$.
The family of interaction energies $E_\omega$ derives from the Lagrangian
\begin{equation}
\label{equation:FKAP}
L(\omega,t) := \frac{1}{2}|t-\lambda|^2  + \frac{K_1}{(2\pi)^2}\big(1-\cos 2\pi\omega_1 \big) + \frac{K_2}{(2\pi)^2} \big( 1 - \cos 2\pi\omega_2 \big),
\end{equation}
where $\omega =(\omega_{1}, \omega_{2}) \in \mathbb{T}^2$.
\end{example}

For an almost periodic interaction model, the notion of {\it ground energy} is given by the following definition.

\begin{definition} \label{definition:GlobalGroundEnergy}
We call ground energy of a family of interactions $\{E_\omega\}_{\omega \in \Omega}$ of Lagrangian form $L:\Omega \times \mathbb{R}^d \to \mathbb{R}$ the quantity
\[
\bar E := \lim_{n\to+\infty} \ \inf_{\omega \in \Omega} \ \inf_{x_0, \ldots,x_n \in \mathbb{R}^d} \ \frac{1}{n}   E_\omega(x_0,\ldots,x_n).
\]
\end{definition}

The above limit is actually a supremum by superadditivity and is finite as soon as $L$ is assumed to be coercive.
Besides, we clearly have \emph{a priori} bounds
\begin{equation}
\inf_{\omega\in\Omega} \ \inf_{x,y \in \mathbb{R}^d} E_\omega(x,y) \leq \bar E \leq \inf_{\omega \in \Omega} \ \inf_{x \in \mathbb{R}^d} E_\omega(x,x).
\end{equation}

The constant $\bar E$ plays the role of a drift. It is natural to modify the previous notion of minimizing configurations by saying that $\{x_n\}_{n\in\mathbb{Z}}$ is {\it calibrated at the level $\bar E$}
if $\sum_{k=m}^{n-1} [E(x_k,x_{k+1}) - \bar E]$ realizes the smallest signed distance between $x_m$ and $x_n$ for every $m<n$. Hence, we consider the following key notions borrowed from the weak KAM theory (see, for instance, \cite{Fathi1997}).

\begin{definition} \label{definition:ManeSubCocycle}
We call Ma\~n\'e potential in the environment $\omega$ the function on $\mathbb{R}^d \times \mathbb{R}^d$ given by
\[
S_\omega(x,y) :=  \inf_{n\geq1} \ \inf_{x=x_0,\ldots,x_n=y} \big[ E_\omega(x_0,\ldots,x_n)-n\bar E \big].
\]
We say that a configuration $\{x_k\}_{k\in\mathbb{Z}}$ is calibrated for $E_\omega$ (at the level $\bar E$) if
\[
\forall \, m < n, \quad S_\omega(x_m,x_n) = E_\omega(x_m,x_{m+1},\ldots,x_n) -(n-m)\bar E.
\]
\end{definition}

As discussed in section \ref{sec:calconf}, the  Ma\~n\'e potential for any almost periodic environment is always finite 
and shares the same properties as a pseudometric. In this context, the calibrated configurations may be seen as geodesics.
An important fact in the framework of almost periodic interaction models 
is that calibrated configurations always exist for some environments~$\omega$.
This is given below by the statement of the first main result of this paper. In section~\ref{sec:calconf}, we introduce minimizing holonomic probabilities, which correspond in our discrete setting to Mather measures, and we define the Mather set as the subset of 
$\Omega\times\mathbb R^d$ formed by the union of their supports (see definition~\ref{definition:HolonomicMeasure}). Denoted $\text{Mather}(L)$, we show that its projection by $pr \colon \Omega \times \mathbb R^d \to \Omega$ is contained into the set of environments for which there exists a calibrated configuration passing through the origin of $\mathbb{R}^d$. Thus, the next theorem extends Aubry-Mather theory of the classical periodic model.

\begin{theorem} \label{theorem:Main1}
Let $(\Omega,\{\tau_t\}_{t\in\mathbb{R}^d},L)$ be an almost periodic interaction model, with $L$ a $C^0$ superlinear function.
Then, for all $\omega\in pr(\text{Mather}(L))$, there exists a calibrated configuration $\{x_k\}_{k\in\mathbb{Z}}$ for $E_\omega$ such that $x_0=0$ and  $\sup_{k\in\mathbb{Z}} \|x_{k+1} - x_k\| < +\infty$.
\end{theorem}

This theorem states that, in the almost periodic case, there exist at least one environment and one calibrated configuration for that environment (and thus for any environment in its orbit).
It may happen that the projected Mather set does not meet every orbit of the system. Indeed, in the almost periodic Frenkel-Kontorova model described  in example \ref{ex:2}, for $\lambda=0$, we have  $\bar E=0$ which is attained by taking $x_n=0$ for every $n\in\mathbb{Z}$. In addition,  it is easy to check that the Mather set is reduced to the point $(0_{\mathbb{T}^2},0_{\mathbb{R}})$ and in particular the projected Mather set $\{0_{\mathbb{T}^2}\}$ meets a unique orbit.
This phenomenon is related to the fact that the minimum of a Lagrangian may not be reached for several orbits of the flow $\tau$ on $\Omega$. A similar case occurs when there is no  exact corrector for the homogenization of Hamilton-Jacobi equations in the stationary ergodic setting \cite{LionsSouganidis,DaviniSiconolfi2009}. At the difference with these works, in our context, we leave open the question whether for any environment there are ``approximated'' calibrated configurations (in a sense to define).

However, we shall see (theorem~\ref{theorem:MainContribution}) that, for a certain class of one-dimensional models with \emph{transversally constant  Lagrangians}, this symptom disappears and the projected Mather set meets any orbit of the system. Such a family of Lagrangians includes the ones of examples~\ref{ex:1} and~\ref{ex:3}. Notice that, for both examples, given any finite configuration, the interaction energy keeps the same value for infinitely many translated configurations. Indeed, take any translation by a multiple of the period in the example \ref{ex:1}, and, for example \ref{ex:3}, take a relatively dense set of translations provided by the collection of return times to the origin of the  irrational rotation of angle $\alpha$ on the circle (recall that this dynamical system is minimal). A transversally constant Lagrangian  is defined in order to share the same property. We postpone to section~\ref{section:Frenkel-Kontorova-quasicrystals},
more precisely, to definition \ref{definition:TransversallyConstant} the details of the technical notion of a transversally constant Lagrangian to be adopted in this article.

Let us also precise that we work on a class of environments more general than the usual one for one-dimensional quasicrystals, in particular,
more general than the one considered in \cite{GambaudoGuiraudPetite2006_01}. Furthermore, we slightly extend the strongly twist 
property~\eqref{equation:StronglyTwistProperty}, which is the main assumption in Aubry-Mather theory (\cite{AubryLeDaeron1983_01, Mather1982}). 
The weakly twist property  will allow  us to use, for example, $\frac{1}{4}|t-\lambda|^4$ instead of $\frac{1}{2}|t-\lambda|^2$ in 
example~\ref{ex:2}, which would be impossible with the strongly twist property \eqref{equation:StronglyTwistProperty}.
We formalize all these extensions  in the next definition.

\begin{definition} \label{definition:quasicrystal_WeaklyTwistProperty} \
Let $(\Omega,\{\tau_t\}_{t\in\mathbb{R}},L)$ be a one-dimensional almost periodic interaction model. 
\begin{itemize}
\item[--] $L$ is said to be weakly twist  if, for every $\omega\in\Omega$, 
$E_\omega(x,y)$ is $C^2$, and
\[
\forall \, x,y\in\mathbb{R}, \ \omega \in \Omega, \quad \frac{\partial^2  E_\omega}{\partial x\partial y}(x,\cdot) < 0 \quad\text{and}\quad
\frac{\partial^2  E_\omega}{\partial x\partial y}(\cdot,y) < 0 \quad\text{ a.e.}
\]
\item[--]  The interaction model $(\Omega,\{\tau_t\}_{t\in\mathbb{R}},L)$ is said to be  almost crystalline if 
\begin{enumerate}
\item $\{\tau_t\}_{t\in\mathbb{R}}$ is uniquely ergodic (with unique invariant probability measure~$\lambda$),
\item $L$ is superlinear and weakly twist,
\item $L$ is  locally transversally constant (as in definition \ref{definition:TransversallyConstant}).
\end{enumerate}
\end{itemize}
\end{definition}

We now state the second main result of this paper, which says that, in the case of almost crystalline interaction models, for any environment, there always exists a calibrated configuration passing close to the origin. This result may be seen as a consequence of the proof of theorem~\ref{theorem:Main1}, since the strategy to obtain it consists in arguing, mainly through a Kakutani-Rohlin tower description of the system 
$(\Omega, \{\tau_{t}\}_{t\in \mathbb R})$, that the corresponding projected Mather set intersects all orbits. A simpler version of the following theorem is given in corollary \ref{corollary:MainContribution}. 

\begin{theorem} \label{theorem:MainContribution}
Let $(\Omega,\{\tau_t\}_{t\in\mathbb{R}},L)$ be an almost crystalline   interaction model. Then the projected Mather set meets uniformly any orbit of the flow $\tau_t$. In particular, for every $\omega \in \Omega$, there exists a calibrated configuration for $E_\omega$ with bounded jumps and at a bounded distance from the origin uniformly in $\omega$, that is, a configuration $\{x_{k,\omega}\}_{k\in\mathbb{Z}}$ satisfying
\begin{enumerate}
\item \label{item:MainContribution_1} $\displaystyle \forall \, m < n, \quad S_\omega(x_{m,\omega},x_{n,\omega}) = \sum_{k=m}^{n-1} E_\omega(x_{k,\omega},x_{k+1,\omega}) -(n-m)\bar E$,
\item \label{item:MainContribution_2} $\displaystyle \sup_{\omega \in \Omega} \ \sup_{k\in\mathbb{Z}} \ |x_{k+1,\omega} - x_{k,\omega}| < +\infty, \qquad \sup_{\omega \in \Omega} \ |x_{0,\omega}| < +\infty$.
\end{enumerate}
\end{theorem}

The paper is organized as follows. In section~\ref{sec:mather}, we properly introduce the Mather set and we show that there exist 
calibrated configurations for almost periodic interaction models by giving the proof of theorem~\ref{theorem:Main1}. Besides, we gather, in 
section~\ref{sec:propcalconf}, ordering properties of one-dimensional calibrated configurations in the presence of the 
twist hypothesis  that will be useful for the demonstration of theorem~\ref{theorem:MainContribution}.
In section~\ref{sec:QCDelone}, we recall basic definitions and properties concerning Delone sets and specially quasicrystals.
In particular, strongly equivariant functions associated with a quasicrystal will serve as a prototype to our notion of locally transversally 
constant Lagrangian. Section~\ref{sec:hull} concerns the basic properties of flow boxes and locally transversally constant Lagrangians. Finally, section~\ref{sec:conf-cal} is devoted to the proof of theorem~\ref{theorem:MainContribution}. 

\section{Calibrated configurations}\label{sec:calconf}

\subsection{Mather set and existence of calibrated configurations}\label{sec:mather}

We show here that the Mather set describes the set of environments for which there exist calibrated configurations.
The Mather set is defined in terms of minimizing holonomic measures.  Let $\underline{\omega} \in \Omega $ be fixed. The ground energy (in the environment $\underline{\omega}$) measures the mean energy per site of a configuration $\{x_n\}_{n\geq0}$ which distributes in $\mathbb{R}^d$ so that $\frac{1}{n}E_{\underline{\omega}}(x_0,\ldots,x_n) \to \bar E$. Notice that the previous mean can be understood as an expectation of $L(\omega,t)$ with respect to a probability measure $\mu_{n,\underline{\omega}} := \frac{1}{n}\sum_{k=0}^{n-1}\delta_{({\tau_{x_k}(\underline{\omega}), \, x_{k+1} - \, x_k})}$:
\begin{equation}
\frac{1}{n} E_{\underline{\omega}}(x_0,\ldots,x_n) = \int\! L(\omega,t) \, \mu_{n,\underline{\omega}}(d\omega,dt).
\end{equation}
Notice also that $\mu_{n,\underline{\omega}}$ satisfies the following property of pseudoinvariance
\begin{equation}
\int\! f(\omega) \, \mu_{n,\underline{\omega}}(d\omega,dt) - \int\! f(\tau_t(\omega)) \, \mu_{n,\underline{\omega}}(d\omega,dt) = \frac{1}{n}\Big(f \circ \tau_{x_n}(\underline{\omega})- f \circ \tau_{x_0}(\underline\omega) \Big).
\end{equation}
This suggests to consider the set of all weak${}^*$ limits of $\mu_{n,\underline{\omega}}$ as $n\to+\infty$. Following~\cite{Mane}, we call these limit measures {\it holonomic probabilities}.

\begin{definition}
\label{definition:HolonomicMeasure}
A probability measure $\mu$ on $\Omega\times\mathbb{R}^d$ is said to be holonomic if
\[
\forall \, f \in C^0(\Omega), \quad \int\! f(\omega) \, \mu(d\omega,dt) = \int\! f(\tau_t(\omega)) \, \mu(d\omega,dt).
\]
Let $\mathbb{M}_{hol}$ denote the set of all holonomic probability measures.
\end{definition}

The set $\mathbb{M}_{hol}$ is certainly not empty since it contains any $\delta_{(\omega,0)}$, $\omega \in \Omega$.  It is then natural to look for holonomic measures that minimize $L$. We show that minimizing holonomic measures do exist and that the lowest mean value of $L$ is the ground energy.

\begin{proposition}[\bf The ergodic formula]
\label{proposition:MatherDefinition}
If $L$ is $C^0$ coercive, then 
\[
\bar E = \inf \Big\{ \int\! L \, d\mu : \mu \in \mathbb{M}_{hol} \Big\}, 
\]
and the infimun is attained by some holonomic probability measure.
\end{proposition}

\begin{definition}
\label{definition:MinimizingMeasure}
We say that an holonomic measure $\mu$ is minimizing if $\bar E = \int\! L \, d\mu$.  We denote by $\mathbb{M}_{min}(L)$ the set of minimizing measures. We call Mather set of $L$ the set
\[
\text{\rm Mather}(L) := \cup_{\mu \in \mathbb{M}_{min}(L)} \textrm{\rm supp}(\mu) \subseteq \Omega \times \mathbb{R}^d.
\]
The projected Mather set is simply $pr(\text{\rm Mather}(L))$, where $pr:\Omega\times\mathbb{R}^d \to \Omega$ is the first projection.
\end{definition}

\begin{proposition} \label{proposition:MatherSet} \
\begin{enumerate}
\item If $L$ is $C^0$ coercive, then
\[
\exists \, \mu \in \mathbb{M}_{min}(L) \quad \text{with} \quad \text{\rm Mather}(L)=\text{\rm supp}(\mu).
\]
In particular, $\text{\rm Mather}(L)$ is closed.

\item If $L$ is $C^0$ superlinear, then $\text{\rm Mather}(L)$ is compact.
\end{enumerate}
\end{proposition}

The set of holonomic measures may be seen as a dual object to the set of coboundaries $\{u - u \circ \tau_t : u \in C^0(\Omega), \ t \in \mathbb{R}^d \}$. Proposition \ref{proposition:MatherDefinition} admits thus a dual version that will be first proved.

\begin{proposition}[\bf The sup-inf formula]
\label{proposition:SupInfFormula}
If $L$ is $C^0$ coercive, then
\begin{align*}
\bar E &= \sup_{u\in C^0(\Omega)} \ \inf_{\omega \in \Omega, \ t\in \mathbb{R}^d} \ \big[ L(\omega,t) + u(\omega) - u \circ \tau_t(\omega) \big].
\end{align*}
\end{proposition}

We do not know whether the above supremum is achieved for general almost periodic interaction models. There is finally a third way to compute the ground energy, which says that the exact choice of the environment $\omega$ is irrelevant.

\begin{proposition}
\label{proposition:LocalGroundEnergy}
If $L$ is $C^0$ coercive, then
\[
\forall \, \omega \in \Omega, \quad \bar E = \lim_{n\to+\infty} \inf_{x_0,\ldots,x_n \in \mathbb{R}^d} \frac{1}{n} E_\omega(x_0,\ldots,x_n).
\]
\end{proposition}

Before proving propositions \ref{proposition:MatherDefinition}, \ref{proposition:SupInfFormula} and \ref{proposition:LocalGroundEnergy}, we note temporarily
\begin{gather*}
\bar E_\omega = \lim_{n\to+\infty} \inf_{x_0,\ldots,x_n \in \mathbb{R}^d} \frac{1}{n} E_\omega(x_0,\ldots,x_n), \quad
\bar L := \inf \Big\{ \int\! L \, d\mu : \mu \in \mathbb{M}_{hol} \Big\}, \\
\text{and} \quad \bar K :=  \sup_{u\in C^0(\Omega)} \ \inf_{\omega \in \Omega, \ t\in \mathbb{R}^d} \ \big[ L(\omega,t) + u(\omega) - u \circ \tau_t(\omega) \big].
\end{gather*}
We first show that the infimum is attained in proposition  \ref{proposition:MatherDefinition}.

\begin{proof}[Proof of proposition \ref{proposition:MatherDefinition}]
 We shall prove later that $\bar L = \bar E$.
We prove now that the infimum is attained in $\bar L := \inf \{\int\!L\,d\mu : \mu \in \mathbb{M}_{hol} \}$.
Let \begin{gather*}
C:=\sup_{\omega\in\Omega} L(\omega,0) \geq \bar L \quad
\text{and} \quad
\mathbb{M}_{hol,C} := \Big\{\mu\in\mathbb{M}_{hol} : \int\!L \,d\mu  \leq C \Big\}.
\end{gather*}
We equip the set of probability measures on $\Omega\times\mathbb{R}^d$ with the weak topology
(convergence of sequence of measures by integration against compactly supported continuous test functions).
By coerciveness, for every $\epsilon>0$ and $M>\inf L$ such that $\epsilon > (C-\inf L)/(M-\inf L)$, there exists
$R(\epsilon)>0$ with $\inf_{\omega\in\Omega, \|t\|\geq R(\epsilon)}L(\omega,t) \geq M$. By integrating $L-\inf L$, we get
\[
\forall \mu\in\mathbb{M}_{hol,C}, \quad\mu\big(\Omega\times\{t : \|t\|\geq R(\epsilon)\}\big) \leq \int\frac{L-\inf L}{M-\inf L} \,d\mu \leq \frac{C-\inf L}{M-\inf L} < \epsilon.
\]
We have just proved that the set $\mathbb{M}_{hol,C}$ is tight. Let $(\mu_n)_{n \geq 0} \subset \mathbb M_{hol,C}$ be a sequence of holonomic measures such that $\int\! L \,d\mu_n\to\bar L$. By tightness, we may assume that $\mu_n \to \mu_\infty$ with respect to the strong topology (convergence of sequence of measures by integration against bounded continuous test functions). In particular, $\mu_\infty$ is holonomic. Moreover, for every $\phi\in C^0(\Omega, [0,1])$, with compact support,
\[
0\leq \int\! (L-\bar L)\phi \,d\mu_\infty = \lim_{n\to+\infty} \int\! (L-\bar L)\phi \, d\mu_n \leq \liminf_{n\to+\infty} \int\!(L-\bar L) \,d\mu_n = 0.
\]
Therefore, $\mu_\infty$ is minimizing.
\end{proof}

We next show that there is no need to take the closure in the definition of the Mather set. We will show later that it is compact.

\begin{proof}[Proof of proposition \ref{proposition:MatherSet} -- Item 1.] We show that $\text{\rm Mather}(L)=\text{\rm supp}(\mu)$ for some minimizing measure $ \mu $. Let $\{V_i\}_{i\in\mathbb{N}}$ be a countable basis of the topology of $\Omega\times\mathbb{R}^d$ and let
\[
I := \{ i \in \mathbb{N} : V_i \cap \text{\rm supp}(\nu) \not= \emptyset \text{ for some }\nu \in \mathbb{M}_{min}(L) \}.
\]
We reindex $I=\{i_1,i_2,\ldots\}$ and choose for every $k\geq1$ a minimizing measure $\mu_k$ so that $V_{i_k} \cap \text{\rm supp}(\mu_k) \not= \emptyset$ or equivalently $\mu_k(V_{i_k})>0$. Let $\mu:=\sum_{k\geq1}\frac{1}{2^k}\mu_k$. Then $\mu$ is minimizing. Suppose some $V_i$ is disjoint from the support of $\mu$. Then $\mu(V_i)=0$ and, for every $k\geq1$, $\mu_k(V_i)=0$. Suppose by contradiction that $ V_i \cap \text{\rm supp}(\nu) \not= \emptyset$ for some $\nu\in\mathbb{M}_{min}(L)$, then $i=i_k$ for some $k\geq1$ and, by the choice of $\mu_k$, $\mu_k(V_i)>0$, which is not possible. Therefore, $V_i$ is disjoint from the Mather set and we have just proved $\text{\rm Mather}(L) \subseteq \text{\rm supp}(\mu)$ or $\text{\rm Mather}(L) = \text{\rm supp}(\mu)$.
\end{proof}

Item 2 of proposition \ref{proposition:MatherSet} will be proved later. We shall need the fact $\Phi=L-\bar L$ on the Mather set, that will be proved in lemma \ref{lemma:CalibratedManeSubCocycle}.

The two formulas given in propositions \ref{proposition:MatherDefinition} and \ref{proposition:SupInfFormula} are two different ways to compute $\bar E$. It is not an easy task to show that the two values are equal. It is the purpose of lemma \ref{lemma:MatherDualApproach} to give a direct proof of this fact. We also give a second proof using the minimax formula.

\begin{lemma}
\label{lemma:MatherDualApproach}
If $L$ is $C^0$ coercive, then $\bar L = \bar K$ and there exists $\mu \in \mathbb{M}_{hol}$ such that $\bar L = \int\! L \, d\mu$.
\end{lemma}

\begin{proof}[First proof of lemma \ref{lemma:MatherDualApproach}]
{\it Part 1.} We show that $\bar L \geq  \bar K$. Indeed, for any holonomic measure $\mu$ and any function $u \in C^0(\Omega)$,
\begin{align*}
\int\! L \,d\mu &= \int\! [L(\omega,t) + u(\omega) - u\circ \tau_t(\omega)] \,\mu(d\omega,dt) \\
&\geq \inf_{\omega \in \Omega, \ t\in \mathbb{R}^d} \ \big[ L(\omega,t) + u(\omega) - u\circ \tau_t(\omega) \big].
\end{align*}
We conclude by taking the supremum on $u$ and the infimum on $\mu$.

{\it Part 2.} We show that $\bar K \geq  \bar L$. Let $X := C_b^0(\Omega\times\mathbb{R}^d)$ be the vector space of bounded continuous functions equipped with the uniform norm. A coboundary is a function $f$ of the form $f=u \circ \tau - u$ or $f(\omega,t)= u \circ \tau_t(\omega) - u(\omega)$ for some $u\in C^0(\Omega)$. Let
\begin{gather*}
A := \{ (f,s) \in X \times \mathbb{R} : \text{ $f$ is a coboundary and } s \geq \bar K \} \quad\text{and} \\
B := \{(f,s)\in X \times \mathbb{R} : \inf_{\omega \in \Omega, \ t\in \mathbb{R}^d} (L-f)(\omega,t) > s \}.
\end{gather*}
Then $A$ and $B$ are nonempty convex subsets of $X\times\mathbb{R}$. They are disjoint by the definition of $\bar K$ and $B$ is open because $L$ is coercive. By Hahn-Banach theorem, there exists a nonzero continuous linear form $\Lambda$ on $X\times\mathbb{R}$ which separates $A$ and $B$. The linear form $\Lambda$ is given by $\lambda \otimes \alpha$, where $\lambda$ is a continuous linear form on $X$ and $\alpha \in \mathbb{R}$. The linear form $\lambda$ is, in particular, continuous on $C^0_0(\Omega \times \mathbb{R}^d)$ and, by Riesz-Markov theorem,
\[
\forall \, f \in C_0^0(\Omega\times\mathbb{R}^d), \quad \lambda(f) = \int\! f \,d\mu,
\]
for some signed measure  $\mu$. By separation, we have
\[
\lambda(f) + \alpha s \leq   \lambda(u - u \circ \tau) + \alpha s',
\]
for $u\in C^0(\Omega)$, $f\in X$ and $s,s'\in \mathbb{R}$ such that $\inf_{\Omega \times \mathbb{R}^d} (L-f)>s$ and $s'\geq \bar K$. By multiplying $u$ by an arbitrary constant, one obtains
\[
\forall \, u \in C^0(\Omega), \quad \lambda(u-u\circ\tau)=0.
\]
The case $\alpha=0$ is not admissible, since otherwise  $\lambda(f)\leq0$ for every $f\in X$ and $\lambda$ would be the null form, which is not possible. The case $\alpha<0$ is not admissible either, since otherwise  one would obtain a contradiction by taking  $f=0$ and $s\to-\infty$. By dividing by $\alpha>0$ and changing $\lambda/\alpha$ to $\lambda$ (as well as $\mu/\alpha$ to $\mu$), one obtains
\[
\forall \, f\in X,   \quad \lambda(f) + \inf_{\Omega\times\mathbb{R}^d} (L-f) \leq \bar K.
\]
By taking $f=c\mathds{1}$, one obtains $c(\lambda(\mathds{1})-1) \leq \bar K - \inf_{\Omega\times\mathbb{R}^d} L$ for every $c\in\mathbb{R}$, and thus $\lambda(\mathds{1})=1$. By taking $-f$ instead of $f$, one obtains $\lambda(f)\geq \inf_{\Omega\times\mathbb{R}^d} L - \bar K $ for every $f\geq0$,
which (again arguing by contradiction) yields $\lambda(f)\geq0$. In particular, $\mu$ is a probability measure. We claim that
\[
\forall \, u \in C^0(\Omega), \quad  \int\!(u-u\circ\tau) \,d\mu=0.
\]
Indeed, given $R>0$, consider a continuous function $ 0\leq\phi_R\leq 1$, with compact support on $\Omega\times B_{R+1}(0)$, such that $ \phi_R \equiv 1 $ on $\Omega\times B_R(0)$. Then
\[
u-u\circ\tau \geq (u-u\circ\tau)\phi_R + \min_{\Omega\times\mathbb{R}^d}(u-u\circ\tau)(1-\phi_R).
\]
Since $\lambda$ and $\mu$ coincide on $C^0_0(\Omega\times\mathbb{R}^d) + \mathbb{R}\mathds{1}$, one obtains
\[
0 = \lambda(u-u\circ\tau) \geq \int\! (u-u\circ\tau)\phi_R \,d\mu + \min_{\Omega\times\mathbb{R}^d}(u-u\circ\tau)\int\! (1-\phi_R) \,d\mu.
\]
By letting $R\to+\infty$, it follows that $ \int\! (u-u\circ\tau) \,d\mu \leq 0$ and the claim is proved by changing $u$ to $-u$. In particular, $\mu$ is holonomic. We claim that
\[
\forall \, f \in X, \quad \int\! f\,d\mu + \inf_{\Omega\times\mathbb{R}^d} (L-f) \leq \bar K.
\]
Indeed, we first notice that the left hand side does not change by adding a constant to $f$. Moreover, if $f\geq 0$ and $0 \leq f_R \leq f$ is any continuous function with compact support on $\Omega\times B_{R+1}(0)$ which is identical to $f$ on $\Omega\times B_R(0)$, the claim follows by letting $R\to+\infty$ in
\[
\int\! f_R\,d\mu + \inf_{\Omega\times\mathbb{R}^d} (L-f) \leq \lambda(f_R) + \inf_{\Omega\times\mathbb{R}^d} (L-f_R) \leq \bar K.
\]
We finally prove the opposite inequality $\bar L \leq  \bar K$. Given $R>0$, denote $L_R = \min(L,R)$. Since $L$ is coercive, $L_R \in X$. Then $L-L_R \geq 0$ and $\int\! L_R \,d\mu \leq \bar K$. By letting $R\to+\infty$, one obtains $\int\! L \,d\mu \leq \bar K$ for some holonomic measure $\mu$.
\end{proof}

We give a second proof of lemma \ref{lemma:MatherDualApproach}. We will use  basic properties of the Kantorovich-Rubinstein topology
on the set of probabilities measures on a Polish space $ (X, d) $ and a version of the Topological Minimax Theorem which is a generalization of Sion's classical result \cite{Sion}. For a recent review on the last topic, see \cite{Tuy}. We state a particular case of theorem 5.7 there.

\begin{theorem}[\bf Topological Minimax Theorem \cite{Tuy}]
\label{theorem:TopologicalMinimaxTheorem}
 Let $ X $ and $ Y $ be Hausdorff topological spaces. Let $F(x,y) :  X \times Y  \to \mathbb{R}$ be a real-valued function. Define $\eta := \sup_{y\in Y} \inf_{x\in X} F(x,y)$ and assume there exists a real number $\alpha^* > \eta$ such that
\begin{enumerate}
 \item \label{item:TopologicalMinimaxTheorem_1} $\forall \, \alpha \in (\eta, \alpha^*)$,  for every finite set $\emptyset \neq H \subset Y$,  $\cap_{y\in H} \{ x \in X : F(x,y) \le \alpha  \} $ is either empty or connected;

\item \label{item:TopologicalMinimaxTheorem_2} $\forall \, \alpha \in (\eta, \alpha^*)$, for every set $K \subset X$,  $\cap_{x\in K} \{ y \in Y : F(x,y) > \alpha \}$ is either empty or connected;

 \item  \label{item:TopologicalMinimaxTheorem_3}  for any $y\in Y$ and $x \in X$, $ F(x,y)$ is lower semi-continuous in $x$ and upper semi-continuous in $y$;

 \item \label{item:TopologicalMinimaxTheorem_4}  there exists  a finite set $M \subset Y$ such that $\cap_{y\in M} \{ x \in X : F(x,y) \le \alpha^*  \}$ is compact and non-empty.
\end{enumerate}
Then,
$$ \inf_{x\in X} \sup_{y \in Y} F(x,y) = \sup_{y \in  Y} \inf_{x\in X} F(x,y).$$
\end{theorem}

We recall basic facts on the Kantorovich-Rubinstein topology (see \cite{Vilani2008} or \cite{AmbrosioGigliSavare}). Given a Polish space $Z$ and  a point $ z_0 \in Z $, let us consider the set of probability measures on the Borel sets of $Z$ that admit a finite first moment, {\em i.e.},
$$ \mathcal P^1(Z) = \big\{ \mu : \int_Z d(z_0, z) \, d\mu(z) < +\infty  \big\}. $$
Notice that this set does not depend on the choice of the point $z_{0}$. The Wasserstein distance or Kantorovitch-Rubinstein distance on $ \mathcal P^1(Z) $ is a distance between two  $\mu, \nu \in \mathcal P^1(Z)$ defined  by
$$ W_1(\mu, \nu) := \inf \big\{  \int_{Z \times Z} d(x,y) \, d\gamma(x,y) : \  \gamma \in \Gamma(\mu, \nu) \big\},$$
where $\Gamma(\mu, \nu)$ denotes the set of all the probability measures $ \gamma $ on $Z \times Z$ with marginals $\mu$ and $\nu$ on the
first and second factors, respectively.

Recall that a continuous function $L \colon Z \to \mathbb R$ is said to be superlinear on a Polish space $ Z $ if the map defined by
$ z \in Z \mapsto L(z)/\big(1+d(z, z_0)\big) \in \mathbb R $ is proper.
Notice that this definition is also independent of the choice of $z_{0}$ and, by considering the distance $\hat{d} :=  \min (d,1)$ on $Z$,
any proper function is superlinear for $\hat{d}$.
The following lemma is easy to prove and gives us a sufficient condition for  relative compactness in $ \mathcal P^1(Z)$ (see theorem 6.9 in \cite{Vilani2008} or \cite{AmbrosioGigliSavare} for a more detailed discussion).

\begin{lemma}
\label{lemma:CompactnessProperty}
Let $Z$ be a Polish space, $L:Z\to\mathbb{R}$ be a  continuous function, and $X:= \{ \mu \in \mathcal{P}^1(Z) \,:\, \int\! L \,d\mu <+\infty \}$ be equipped with the Kantorovich-Rubinstein distance. Then
\begin{enumerate}
\item  the map $\mu \in X \mapsto \int L \, d\mu$ is lower semi-continuous;
\item  if $L$ is a superlinear, then, for every $\alpha \in \mathbb R$, the set $\{ \mu \in X : \int\! L \, d\mu \le \alpha \}$ is compact 
(the map $\mu \in X \mapsto \int L \, d\mu$ is proper).
\end{enumerate}
\end{lemma}

\begin{proof}[Second proof of lemma \ref{lemma:MatherDualApproach}]
Lemma \ref{lemma:CompactnessProperty} applied  to the $ C^0 $ superlinear Lagrangian $L: \Omega \times \mathbb R^d \to \mathbb{R}$  guarantees the existence of a minimizing  probability for $L$. This minimizing measure is holonomic since  the set of holonomic measures is a closed subset of $ \mathcal P^1(\Omega \times \mathbb R^d) $ for the
Kantorovich-Rubinstein distance.
Notice that for every $ u \in C^0(\Omega)$,
\begin{eqnarray*}
\inf_{\omega \in \Omega, \ t \in \mathbb{R}^d} (L+u-u\circ\tau)(\omega,t) &=&  \inf_{\omega \in \Omega, \ t \in \mathbb R^d} \int (L + u - u \circ \tau) \, d\delta_{(\omega,t)} \\
& \ge & \inf_{\mu \in \mathcal P^1(\Omega \times \mathbb R^d)} \int (L + u - u \circ \tau) \, d\mu \\
& \ge & \inf_{\omega \in \Omega, \ t \in \mathbb R^d} (L + u - u \circ \tau) (\omega, t).
\end{eqnarray*}
Let $X:=\{ \mu \in \mathcal{P}^1(\Omega \times \mathbb{R}^d) \,:\, \int\! L \, d\mu < +\infty \}$ and $Y := C^0(\Omega)$. Then
\begin{equation*}
\bar K = \sup_{u \in Y} \inf_{\mu \in X} \int (L + u - u \circ \tau) \, d\mu \leq \min_{\omega \in \Omega} L(\omega,0).
\end{equation*}
Define  $\alpha^* := \min_{\omega \in \Omega}L(\omega,0)+1 >  \bar K$ and
\[
F : (\mu, u) \in  X \times Y\mapsto \int (L + u - u \circ \tau) \, d\mu.
\]
Since $F$ is affine in both variables, it satisfies items \ref{item:TopologicalMinimaxTheorem_1} and \ref{item:TopologicalMinimaxTheorem_2} of theorem \ref{theorem:TopologicalMinimaxTheorem}. Item~\ref{item:TopologicalMinimaxTheorem_3} is also satisfied since $F(\mu,u)$ is lower semi-continuous in $\mu$ and continuous in $u$. By taking $M = \{0\}$, the singleton set reduced to the null function in $Y$,  the set $\cap_{u\in M} \{ \mu \in X \,:\, F(\mu,u) \leq \alpha^* \}$ is compact and non-empty, so that item \ref{item:TopologicalMinimaxTheorem_4} is satisfied. The Topological Minimax Theorem therefore implies 
\begin{align}\label{eq:barK}
\bar{K} = \inf_{ \mu \in X}  \sup_{u \in Y} \int (L + u - u\circ\tau) \, d\mu.
\end{align}
We show that every $\mu \in X$  such that $\sup_{u \in Y} \int (L + u - u\circ\tau) \, d\mu < +\infty$ is holonomic. If  not, there would exist a function $u \in C^0(\Omega)$ such that $ \int (u - u \circ \tau) \, d\mu >0$. Multiplying $(u - u \circ \tau)$ by a positive scalar $\lambda$ and letting $\lambda \to +\infty$ would lead to a contradiction.  Thus, the infimum in~(\ref{eq:barK}) may be taken over holonomic probabilty measures with respect to which $ L $ is integrable. We finally conclude that
\begin{equation*}
\bar{K} =  \inf_{ \mu \in X}  \sup_{u \in Y} \int (L + u - u\circ\tau) \, d\mu = \inf_{\mu  \in \mathbb M_{hol}} \int L \, d\mu = \bar{L}. \qedhere
\end{equation*}
\end{proof}

The holonomic condition shall not be confused with invariance in the usual sense of dynamical systems. We may nevertheless introduce a larger space than $\Omega\times\mathbb{R}^d$ and a suitable dynamics on such a space. We will apply Birkhoff ergodic theorem with respect to that dynamical system to prove that $\bar L \geq \bar E$.

\begin{notation}
\label{notation:MarkovExtension}
Consider $\hat\Omega := \Omega \times (\mathbb{R}^d)^{\mathbb{N}}$ equipped with the product topology and the Borel sigma-algebra. 
In particular, $\hat\Omega$ becomes a complete separable metric space. Any probability measure $\mu$ on $\Omega \times \mathbb{R}^d$ admits a unique disintegration along the first projection $pr: \Omega\times\mathbb{R}^d \to \Omega$,
\[
\mu(d\omega,dt) := pr_*(\mu)(d\omega) P(\omega,dt),
\]
where $\{P(\omega,dt)\}_{\omega\in\Omega}$ is a measurable family of probability measures on $\mathbb{R}^d$. Let $\hat\mu$ be the Markov measure with initial distribution $pr_*(\mu)$ and transition probabilities $P(\omega,dt)$. For Borel bounded functions of the form $f(\omega,t_0,\ldots,t_n)$, we have
\begin{equation*}
\hat\mu(d\omega,d\underline t)  = pr_*(d\omega)P(\omega,dt_0)P(\tau_{t_0}(\omega),dt_1)\cdots P(\tau_{t_0+\cdots+t_{n-1}}(\omega),dt_n).
\end{equation*}
If $\mu$ is holonomic, then $\hat\mu$ is invariant with respect to the shift map
\[
\hat\tau : (\omega,t_0,t_1,\ldots) \mapsto (\tau_{t_0}(\omega), t_1,t_2, \ldots).
\]
We will call $\hat\mu$ the Markov extension of $\mu$. Conversely, the projection of any $\hat \tau$-invariant probability measure $\tilde\mu$ on $\Omega\times\mathbb{R}^d$ is holonomic. This gives a fourth way to compute $\bar E$
\[
\bar E = \inf\Big\{ \int\! \hat L \,d\tilde\mu : \tilde\mu \text{ is a $\hat\tau$-invariant probability measure on $\hat\Omega$} \Big\},
\]
where $\hat L(\omega,t_0,t_1,\ldots) := L(\omega,t_0)$ is the natural extension of $L$ on $\hat \Omega$.
\end{notation}

\begin{proof}[End of proof of propositions \ref{proposition:MatherDefinition}, \ref{proposition:SupInfFormula} and \ref{proposition:LocalGroundEnergy}] $ $

-- Part 1: We know that $\bar K = \bar L$ by lemma \ref{lemma:MatherDualApproach}.

\smallskip

-- Part 2: We claim that $\bar E_\omega = \bar E$ for all $ \omega \in \Omega $. By the topological stationarity~(\ref{equation:Equivariance}) of $E_\omega$
and by the minimality of $\tau_t$, for any $ n \in \mathbb N $, we have that
{\begin{eqnarray*}
\inf_{x_0,\ldots,x_n \in \mathbb{R}^d} \ E_\omega(x_0,\ldots,x_n)  &=&\inf_{x_0,\ldots,x_n \in \mathbb{R}^d} \ \inf_{t\in\mathbb{R}^d} \   E_{\omega}(x_0+t,\ldots,x_n+t) \\
 & = & \inf_{x_0,\ldots,x_n \in \mathbb{R}^d} \ \inf_{t\in\mathbb{R}^d} \ E_{\tau_t(\omega)}(x_0,\ldots,x_n) \\
& = & \inf_{x_0,\ldots,x_n \in \mathbb{R}^d} \ \inf_{\omega \in \Omega}  \ E_\omega(x_0,\ldots,x_n),
\end{eqnarray*}}
which clearly yields $ \bar E_\omega = \bar E$ for every $ \omega \in \Omega $.

\smallskip

-- Part 3: We claim that $\bar E \geq \bar K$. Indeed, given $c < \bar K$, let $u\in C^0(\mathbb{R}^d)$ be such that, for every $\omega\in \Omega$ and any $t\in\mathbb{R}^d$, $u(\tau_t(\omega))-u(\omega) \leq L(\omega,t)-c$. Define $u_\omega(x) = u(\tau_x(\omega))$. Then,
\[
\forall \, x,y \in\mathbb{R}^d, \quad u_\omega(y)-u_\omega(x) \leq E_\omega(x,y)-c,
\]
which implies $\bar E  \geq c$ for every $c<\bar K$, and therefore $\bar E \geq \bar K$.

\smallskip

-- Part 4: We claim that $\bar L \geq  \bar E$. Let $\mu$ be a minimizing holonomic probability measure with Markov extension $\hat\mu$ (see notation \ref {notation:MarkovExtension}). If $(\omega,\underline{t})\in\hat\Omega$, then
\[
\sum_{k=0}^{n-1} \hat L\circ\hat\tau^k(\omega,\underline{t}) = E_\omega(x_0,\ldots,x_n) \quad\text{with}\quad x_0 = 0 \text{ and } x_k = t_0+\cdots+t_{k-1},
\]
and, by Birkhoff ergodic theorem,
\[
\bar E \leq \int\! \lim_{n\to+\infty} \frac{1}{n} \sum_{k=0}^{n-1} \hat L \circ \hat\tau^k \,d\hat\mu = \int\! L \,d\mu = \bar L. \qedhere
\]
\end{proof}

A backward calibrated sub-action $u$ as given by the Lax-Oleinik operator in the periodic context (for details, see~\cite{GaribaldiThieullen2011_01}) is not available in general for an almost periodic interaction model. 
A calibrated sub-action $u$ in this setting would be  a $C^0(\Omega)$ function such that, if $E_{\omega,u}$ is defined by
\begin{equation*}
E_{\omega,u}(x,y) := E_\omega(x,y) - \big[ u\circ\tau_{y}(\omega) - u\circ\tau_{x}(\omega) \big] - \bar E,
\end{equation*}
then
\begin{gather*}
\left\{\begin{array}{l}
\forall \, \omega \in\Omega, \ \forall \, x,y \in \mathbb{R}^d, \quad  E_{\omega,u}(x,y) \geq 0,  \\
\forall \, \omega \in\Omega, \ \forall \, y \in \mathbb{R}^d, \ \exists \, x \in \mathbb{R}^d, \quad E_{\omega,u}(x,y)=0.
\end{array}\right.
\end{gather*}
We do not know whether such a function exists. We will weaken this notion by introducing a notion of measurable subadditive cocycle. 
Notice first  that the function $U(\omega,t) := u\circ\tau_t(\omega) - u(\omega)$ is a cocycle, namely, it satisfies
\begin{equation}
\forall \, \omega \in \Omega, \ \forall \, s,t \in\mathbb{R}^d, \quad U(\omega,s+t) = U(\omega,s) + U(\tau_s(\omega),t).
\end{equation}
A natural candidate to be a subadditive function is given by the Ma\~n\'e potential in the periodic context. 
For almost periodic interaction models, we introduce the following definition.

\begin{definition} \label{definition:ManeSubCocycle2}
Let $L$ be a coercive Lagrangian. We call Ma\~n\'e subadditive cocycle associated with  $L$ the function defined on $\Omega \times \mathbb{R}^d$ by
\[
\Phi(\omega,t) := \inf_{n\geq 1} \ \inf_{0=x_0,x_1,\ldots,x_n=t} \ \sum_{k=0}^{n-1} \big[ L(\tau_{x_k}(\omega),x_{k+1}-x_k) - \bar E \big].
\]
We call Ma\~n\'e potential in the environment $\omega$ the function on $\mathbb{R}^d \times \mathbb{R}^d$ given by
\[
S_\omega(x,y) := \Phi(\tau_x(\omega),y-x) = \inf_{n\geq1} \ \inf_{x=x_0,\ldots,x_n=y} \big[ E_\omega(x_0,\ldots,x_n)-n\bar E \big].
\]
\end{definition}

The very definitions of $\Phi$ and $\bar E$ show that $\Phi$ takes finite values and is a subadditive cocycle,
\begin{align}
&\forall \, \omega \in \Omega, \ \forall \, s,t \in \mathbb{R}, \quad \Phi(\omega,s+t) \leq \Phi(\omega,s) + \Phi(\tau_s(\omega),t), \\
&\forall \, \omega \in \Omega, \ \forall \, t \in \mathbb{R}^d, \quad \Phi(\omega,t) \leq L(\omega,t) - \bar E,\\
&\forall \, \omega \in \Omega, \quad \Phi(\omega,0)\geq0, \label{eq:SubadditiveCocycle_3}\\
&\forall \, \omega \in \Omega, \ \forall \, t \in \mathbb{R}^d, \quad \Phi(\omega, t) \ge \bar E - L(\tau_t(\omega), -t).
\end{align}
Inequality \eqref{eq:SubadditiveCocycle_3} is proved using the fact that, for a fixed $\omega$, the sequence
\[
\bar E_n(\omega,0) := \inf_{x_1,\ldots,x_{n-1}}E_\omega(0,x_1,\ldots,x_{n-1},0)
\] 
is subadditive in $n$ and $ \bar E \le \lim_{n\to\infty}\frac{1}{n} \bar E_n(\omega,0) = \inf_{n\geq1} \frac{1}{n} \bar E_n(\omega,0)$.

We will prove in addition that $\Phi$ is upper semi-continuous and Mather-calibrated (lemma \ref{lemma:CalibratedManeSubCocycle})  in the following sense.

\begin{definition} \label{definition:CalibratedSubCocycle}
A measurable function $U:\Omega \times \mathbb{R}^d \to [-\infty,+\infty[$ is called a  Mather-calibrated subadditive cocycle if the following properties are satisfied:
\par-- $\forall \, \omega \in \Omega, \ \forall \, s,t \in\mathbb{R}^d, \quad U(\omega,s+t) \leq U(\omega,s) + U(\tau_s(\omega),t)$,
\par-- $\forall \, \omega \in \Omega, \ \forall \, s,t \in\mathbb{R}^d, \quad U(\omega,t) \leq L(\omega,t) - \bar L \quad \text{and} \quad U(\omega,0) \geq 0$,
\par-- $\forall \, \mu \in\mathbb{M}_{hol}, \, \text{ if } \,  \int\! L \,d\mu < +\infty, \ \text{ then } \ \int\! U(\omega,\sum_{k=0}^{n-1}t_k) \, \hat\mu(d\omega,d\underline t) \geq 0, \,\,\, \forall \, n \geq 1 $,
\par{\color{white}--} where $ \hat \mu $ is the Markov extension of $ \mu $.
\end{definition}

Notice that, provided we know in advance that $U$ is finite,  $U(\omega,0)\geq0$ by replacing $s=t=0$ in the subadditive cocycle inequality.

\begin{lemma}
\label{lemma:CalibratedSubCocycle}
A Mather-calibrated subadditive cocycle $U$ satisfies in addition
\par-- $U(\omega,t)$ is finite everywhere,
\par-- $ \ \sup_{\omega \in \Omega, t\in\mathbb{R}^d} \ |U(\omega,t)| /(1+\|t\|) < +\infty$,
\par-- $\forall \, \mu \in \mathbb{M}_{min}(L), \,\, \forall \, n\geq1, \quad U(\omega,\sum_{k=0}^{n-1}t_k) = \sum_{k=0}^{n-1} [\hat L - \bar L] \circ \hat\tau^k (\omega,\underline{t}) \ $ $ \ \hat\mu $ a.e.
\end{lemma}

\begin{proof} {\it Part 1.}
We show that $U$ is sublinear. Let $K := \sup_{\omega \in \Omega, \ \|t\| \leq 1} [L(\omega,t) - \bar L]$. Fix $t\in\mathbb{R}^d$ and choose the unique integer $n$ such that $n-1 \leq \|t\| < n$.  Let $t_k = \frac{k}{n}t$ for $ k = 0, \ldots, n - 1 $. Then the subadditive cocycle property implies, on the one hand,
\[
\forall \, \omega \in\Omega,  \ \forall \, t \in \mathbb{R}^d, \quad U(\omega,t) \leq \sum_{k=0}^{n-1} U(\tau_{t_k}(\omega), t_{k+1}-t_k) \leq nK \leq (1+\|t\|) K.
\]
On the other hand, thanks to the hypothesis $U(\omega,0)\geq0$, we get the opposite inequality
\[
\forall \, \omega \in\Omega,  \ \forall \, t \in \mathbb{R}^d, \quad U(\omega,t) \geq U(\omega,0) - U(\tau_t(\omega),-t) \geq -(1+\|t\|)K.
\]
We also have shown that $U$ is finite everywhere.

\medskip
\noindent{\it Part 2.} Suppose $\mu$ is minimizing. Since
\[
\forall \, \omega\in\Omega, \ \forall \, t_0,\ldots,t_{n-1} \in \mathbb{R}^d, \quad\sum_{k=0}^{n-1}  \big[\hat L - \bar L \big] \circ \hat\tau^k (\omega,\underline{t}) \geq U \Big(\omega,\sum_{k=0}^{n-1}t_k \Big),
\]
by integrating with respect to $\hat\mu$, the left hand side has a null integral whereas the right hand side has a nonnegative integral. The previous inequality is thus an equality that holds almost everywhere.  \qedhere
\end{proof}

\begin{lemma}
\label{lemma:CalibratedManeSubCocycle}
If $L$ is $C^0$ coercive, then the Ma\~n\'e subadditive cocycle $\Phi$ is upper semi-continuous and Mather-calibrated. In particular, $\Phi = L - \bar L$ on $\text{\rm Mather}(L)$, or more precisely, for every $\mu \in \mathbb{M}_{min}(L)$, being $\hat \mu$ its Markov extension,
\begin{align*}
 \forall \, (\omega,\underline{t}) \in \text{\rm supp}(\hat\mu), \ \forall \, i<j, \quad \Phi \Big( \tau_{\sum_{k=0}^{i-1} t_k} (\omega),\sum_{k=i}^{j-1} t_k \Big) = \sum_{k=i}^{j-1} \big[L-\bar L\big] \circ \hat \tau^k(\omega,\underline{t}).
\end{align*}
In an equivalent manner, if $(\omega,\underline{t}) \in \text{\rm supp}(\hat\mu)$,  $x_0=0$ and $x_{k+1}=x_k+t_k$, $\forall \, k\geq0$, the semi-infinite configuration $\{x_k\}_{k\geq0}$ is calibrated for $E_\omega$ as in {definition} \ref{definition:ManeSubCocycle}:
\[
\forall  \, i<j, \quad  S_\omega(x_i,x_j) = E_\omega(x_i,x_{i+1}, \ldots, x_j) -(j-i) \bar E.
\]
\end{lemma}

\begin{proof}
{\it Part 1.} We first show the existence of a particular measurable Mather-calibrated subadditive cocycle $U(\omega,t)$. From the sup-inf formula (proposition \ref{proposition:SupInfFormula}), for every $p\geq1$, there exists $u_p \in C^0(\Omega)$ such that
\[
\forall \, \omega \in \Omega, \ \forall \, t\in\mathbb{R}^d, \quad u_p \circ \tau_t(\omega) - u_p(\omega) \leq L(\omega,t) - \bar L + 1/p.
\]
Let $U_p(\omega,t) := u_p\circ\tau_t(\omega) - u_p(\omega)$ and $U := \limsup_{p\to+\infty}U_p$. Then $U$ is clearly a subadditive cocycle and satisfies $U(\omega,0)=0$. Besides, $ U $ is finite everywhere, since $ 0 = U(\omega, 0) \le U(\omega, t) + U(\tau_t(\omega), -t) $ and $ U(\omega, t) \le L(\omega, t) - \bar L $. We just verify the last property in definition \ref{definition:CalibratedSubCocycle}. Let $\mu \in \mathbb{M}_{hol}$ be such that $\int\! L \,d\mu < +\infty$. For $n\geq1$, let
\[
\hat S_{n,p}(\omega,\underline t) := \sum_{k=0}^{n-1} \Big[\hat L- \bar L +\frac{1}{p} \Big] \circ \hat\tau^k (\omega,\underline t)  -U_p \Big(\omega, \sum_{k=0}^{n-1} t_k \Big) \geq 0.
\]
Since
\[
U_p \Big(\omega,\sum_{k=0}^{n-1}t_k \Big) = \sum_{k=0}^{n-1} \hat U_p \circ  \hat\tau^k(\omega,\underline{t}), \quad \hat U_p(\omega,\underline{t}) := U_p(\omega,t_0),
\]
by integrating with respect to $\hat\mu$, we obtain
\[
0 \leq  \int\! \inf_{p\geq q} \hat S_{n,p} \, d\hat\mu \leq \inf_{p\geq q}  \int\! \hat S_{n,p}(\omega,\underline t) \, d\hat\mu \leq n\int\! \Big[ L - \bar L + \frac{1}{q} \Big] \,d\mu.
\]
By Lebesgue's monotone convergence theorem, as $q\to+\infty$, we have
\begin{gather*}
\int\! \Big[n(\hat L - \bar L) - U\Big(\omega, \sum_{k=0}^{n-1} t_k \Big) \Big] d\hat\mu \leq  \int\! n[ L - \bar L ] \,d\mu \quad \text{and} \\
\int\! U\Big(\omega, \sum_{k=0}^{n-1} t_k \Big) \,\hat\mu(d\omega,d\underline{t}) \geq0.
\end{gather*}

\noindent{\it Part 2.} We next show that $\Phi$ is Mather-calibrated. We have already noticed that $\Phi$ satisfies the subadditive cocycle property, $\Phi \leq L - \bar L$, $\Phi(\omega,0)\geq0$, and  $\Phi(\omega,t)$ is finite everywhere. Moreover, $\Phi(\omega,t) \geq U(\omega,t)$ and the third property of definition~\ref{definition:CalibratedSubCocycle} follows from part 1.

\medskip
\noindent{\it Part 3.} We show that $\Phi$ is upper semi-continuous. For $n\geq1$, let
\begin{gather*}
S_n(\omega,t) := \inf \{ E_\omega(x_0,\ldots,x_n) : x_0=0, \ x_n=t \}.
\end{gather*}
Then $\Phi = \inf_{n\geq1} (S_n  - n\bar E) $ is upper semi-continuous if we prove that $S_n(\omega,t)$ is continuous on every bounded set
with $\omega \in \Omega$ and $\|t\| \leq D$. Denote $c_0 := \inf_{\omega,x,y}E_\omega(x,y)$ and 
$K := \sup_{\omega \in \Omega, \ \|t\| \leq D} E_\omega(0,\ldots,0,t)$. By coerciveness, there exists $R>0$ such that
\[
\forall \, x,y \in \mathbb{R}^d, \quad\|y-x\| > R \Rightarrow \ \forall \, \omega \in \Omega, \ E_\omega(x,y) > K - (n-1) c_0.
\]
Suppose $\omega,x_0,\ldots,x_n$ are such that $E_\omega(x_0,\ldots,x_n) \leq K$. Suppose by contradiction that $\|x_{k+1}-x_k\| > R$ for some $k\geq0$. Then
\[
K \geq E_\omega(x_0,\ldots,x_n) \geq (n-1)c_0 + E_\omega(x_k,x_{k+1}) > K,
\]
which is impossible. We have proved that the infimum in the definition of $S_n(\omega,t)$, for every $\omega\in\Omega$ and $\|t\| \leq D$, can be realized
by some points $\|x_k\| \leq kR$. By the uniform continuity of $E_\omega(x_0,\ldots,x_n)$ on the product space $\Omega \times \Pi_k \{\|x_k\| \leq kR \}$,
we obtain that $S_n$ is continuous on $\Omega \times  \{ \|t\| \leq D \}$.

\medskip
\noindent{\it Part 4.} Let $\mu$ be a minimizing measure with Markov extension $ \hat\mu $. We show that every $(\omega,\underline{t})$ in the support of $\hat\mu$ is calibrated. Let
\[
\hat \Sigma := \Big\{ (\omega,\underline{t}) \in \Omega \times (\mathbb{R}^d)^{\mathbb{N}} : \forall \, n\geq1, \  \Phi \Big( \omega,\sum_{k=0}^{n-1} t_k \Big) \geq \sum_{k=0}^{n-1} \big[L-\bar L\big] \circ \hat \tau^k(\omega,\underline{t}) \Big\}.
\]
The set $\hat\Sigma$ is closed, since $\Phi$ is upper semi-continuous. By lemma \ref{lemma:CalibratedSubCocycle}, $\hat\Sigma$ has full $\hat\mu$-measure and  therefore contains $\text{\rm supp}(\hat\mu)$. Thanks to the subadditive cocycle property of $\Phi$ and the $\hat\tau$-invariance of
$\text{\rm supp}(\hat\mu)$, we obtain the calibration property
\[
\forall \, (\omega,\underline{t}) \in \hat\Sigma, \ \forall \, 0 \leq i<j, \quad \Phi \Big( \tau_{x_i}(\omega),\sum_{k=i}^{j-1} t_k \Big) = \sum_{k=i}^{j-1} \big[L-\bar L\big] \circ \hat \tau^k(\omega,\underline{t}). \qedhere
\]
\end{proof}

\begin{proof}[Proof of proposition \ref{proposition:MatherSet} -- Item 2.] We now assume that $L$ is superlinear. From lemma~\ref{lemma:CalibratedSubCocycle}, the Ma\~n\'e subadditive cocycle is at most linear. There exists $R>0$ such that
\[
\forall \, \omega \in \Omega, \ \forall \, t \in \mathbb{R}^d, \quad |\Phi(\omega,t)| \leq R(1+\|t\|).
\]
By superlinearity, there exists $B>0$ such that
\[
\forall \, \omega \in \Omega, \ \forall \, t \in \mathbb{R}^d, \quad L(\omega,t) \geq 2R\|t\| -B.
\]
Let $\mu$ be a minimizing measure. Since $\Phi=L-\bar L$  $\mu$ a.e. (lemma \ref{lemma:CalibratedSubCocycle}), we obtain
\[
\|t\| \leq (R+B+|\bar L|)/R, \quad \mu(d\omega,dt) \text{ a.e.}
\]
We have proved that the support of every minimizing measure is compact. In particular, the Mather set is compact.
\end{proof}

\begin{proof}[Proof of theorem \ref{theorem:Main1}.] We show that, for every environment $\omega$ in the projected Mather set, there exists a calibrated configuration for $E_\omega$ passing through the origin. Let $\mu$ be a minimizing measure such that $\text{\rm supp}(\mu)=\text{\rm Mather}(L)$. Let $\hat \mu$ denote its Markov extension. For $n\geq1$, consider
\[
\hat\Omega_n := \Big\{ (\omega,\underline{t}) \in \Omega \times (\mathbb{R}^d)^{\mathbb{N}} : \Phi \Big( \omega,\sum_{k=0}^{2n-1} t_k \Big) \geq \sum_{k=0}^{2n-1} \big[L-\bar L\big] \circ \hat \tau^k(\omega,\underline{t}) \Big\}.
\]
From lemma \ref{lemma:CalibratedManeSubCocycle},  $\text{\rm supp}(\hat\mu) \subseteq \hat\Omega_n$. From the upper semi-continuity of $\Phi$, $\hat\Omega_n$ is closed. To simplify the notations, for every $\underline{t}$, we define a configuration $(x_0,x_1,\ldots)$ by
\[
x_0 = 0, \ x_{k+1} = x_k + t_k \quad\text{so that}\quad
\hat\tau^k(\omega,\underline{t}) = (\tau_{x_k}(\omega),(t_k,t_{k+1}, \ldots)).
\]
Notice that, if $(\omega,\underline{t}) \in \hat\Omega_n$, thanks to the subadditive cocycle property of $\Phi$ and the fact that $ \Phi \le L - \bar L $, the finite configuration $(x_0,\ldots,x_{2n})$ is calibrated in the environment $\omega$, that is,
\[
\forall \, 0 \leq i<j \leq 2n, \quad \Phi \Big( \tau_{x_i}(\omega),\sum_{k=i}^{j-1} t_k \Big) = \sum_{k=i}^{j-1} \big[L-\bar L\big] \circ \hat \tau^k(\omega,\underline{t}),
\]
or written using the family of interaction energies $E_\omega$,
\[
\forall \, 0 \leq i<j \leq 2n, \quad S_\omega(x_i,x_j) = E_\omega(x_i,\ldots,x_j) -(j-i)\bar E.
\]
Thanks to the sublinearity of $S_\omega$, there exists a constant $R>0$ such that, uniformly in $\omega \in \Omega $ and $x, y \in \mathbb R^d$, we have $|S_\omega(x,y)| \leq R(1+\|y-x\|)$. Besides, thanks to the superlinearity of $E_\omega$, there exists a constant $B>0$ such that $E_\omega(x,y) \geq 2R \|y-x\| - B$. Since $S_\omega(x_k,x_{k+1}) = E_\omega(x_k,x_{k+1}) - \bar E$, we thus obtain a uniform upper bound $D:=(R+B+|\bar E|)/R$ on the jumps of calibrated configurations:
\[
\forall \, (\omega,\underline{t}) \in \hat\Omega_n, \ \forall \, 0 \leq k < 2n, \quad \|x_{k+1} - x_k \| \leq D.
\]
Let $\hat\Omega'_n = \hat\tau^n(\hat\Omega_n)$. Thanks to the uniform bounds on the jumps, $\hat\Omega'_n$ is again closed.  Since $\hat\mu(\hat\Omega_n)=1$, $\hat\mu(\hat\Omega'_n)=1$ by invariance of $\hat \tau$. Let $\nu := pr_*(\mu)$ be the projected measure on $\Omega$. Then $\text{\rm supp}(\nu)=pr(\text{\rm Mather}(L))$. By the definition of $\hat\Omega'_n$, we have
\begin{multline*}
\hat{pr}(\hat\Omega'_n) = \{ \omega \in \Omega : \exists \, (x_{-n},\ldots,x_n)\in\mathbb{R}^d \ \text{ s.t. } \ x_0=0 \ \text{ and}
\\ S_\omega(x_{-n},x_{n}) \geq E_\omega(x_{-n},\ldots,x_n) -2n\bar E \}.
\end{multline*}
Again by compactness of the jumps, $\hat{pr}(\hat\Omega'_n)$ is closed and has full $\nu$-measure. Thus, $\hat{pr}(\hat\Omega'_n) \supseteq pr(\text{\rm Mather}(L))$. By a diagonal extraction procedure, we obtain, for every $\omega \in \text{\rm Mather}(L)$, a bi-infinite calibrated configuration with uniformly bounded jumps passing through the origin.
\end{proof}

\subsection{Properties of one-dimensional calibrated configurations}\label{sec:propcalconf}

Perhaps the most powerful assumption made in one-dimensional Aubry theory \cite{AubryLeDaeron1983_01} is the twist property. It will not be used here in the infinitesimal form. Supposing that $ (\Omega, \{\tau_t\}_{t \in \mathbb R}, L) $ is weakly twist (definition~\ref{definition:quasicrystal_WeaklyTwistProperty}), we discuss in this section key properties on the ordering of minimizing configurations and therefore of calibrated configurations. The fundamental Aubry crossing property is explained in lemma~\ref{lemma:twistProperty}.
We collect in lemmas~\ref{lemma:twistMonotony} and \ref{lemma:StrictlyMonotone} 
intermediate results, that are consequences of the weakly twist property, about the order of the points composing a minimizing configuration.
Such results will be applied in the proof of theorem~\ref{theorem:MainContribution} in section~\ref{sec:conf-cal}.

The following lemma is an easy consequence of the definition. It shows that the energy of a configuration can be lower by exchanging  the positions.

\begin{lemma}[\bf Aubry crossing lemma]
\label{lemma:twistProperty}
If $L$ is weakly twist, then, for every $\omega \in \Omega$, for every $x_0,x_1,y_0,y_1 \in \mathbb{R}$ satisfying $(y_0-x_0)(y_1-x_1) < 0$,
\[
\big[ E_\omega(x_0,x_1)+E_\omega(y_0,y_1) \big]  - \big[ E_\omega(x_0,y_1)+E_\omega(y_0,x_1) \big] = \alpha(y_0-x_0)(y_1-x_1) >0,
\]
with $ \alpha = \frac{1}{(y_0-x_0)(y_1-x_1)} \int_{x_0}^{y_0} \int_{x_1}^{y_1} \frac{\partial^2   E_\omega}{\partial x\partial y} (x,y) \, dy dx < 0 $.
\end{lemma}

\begin{proof}
The inequality is obtained by integrating the function $\frac{\partial^2}{\partial x \partial y}  E_\omega$ on the domain
$ [\min(x_0,y_0),\max(x_0,y_0)]\times[\min(x_1,y_1),\max(x_1,y_1)] $.
\end{proof}

The first consequence of Aubry crossing lemma is that minimizing configurations shall be strictly ordered. We begin by an intermediate lemma.

\begin{lemma}
\label{lemma:twistMonotony}
Let $L$ be a weakly twist Lagrangian, $\omega \in \Omega$, $n\geq2$, and $x_0,\ldots,x_n \in \mathbb{R}$ be a nonmonotone sequence
(that is, a sequence which does not satisfy $x_0\leq\ldots\leq x_n$ nor $x_0\geq\ldots\geq x_n$).

-- If $x_0=x_n$, then $E_\omega(x_0,\ldots,x_n) > \sum_{i=0}^{n-1} E_\omega(x_i,x_i)$.

-- If $x_0\not= x_n$, then there exists a subset $\{i_0,i_1,\ldots,i_r\}$ of $\{0,\ldots,n\}$, with $i_0=0$ and $i_r=n$, such that $(x_{i_0},x_{i_1},\ldots,x_{i_r})$ is strictly monotone and
\[
E_\omega(x_0,\ldots,x_n) > E_\omega(x_{i_0},\ldots,x_{i_r}) + \sum_{ i\not\in\{i_0,\ldots,i_r\}} E_\omega(x_i,x_i).
\]
(Notice that it may happen that $x_i = x_j$ for  $i\not\in \{i_0,\ldots,i_r\}$ and $j \in \{i_0,\ldots,i_r\}$.)

\end{lemma}

\begin{proof}

We prove the lemma by induction.

Let $x_0,x_1,x_2\in\mathbb{R}$ be a nonmonotone sequence. If $x_0=x_2$, then $E_\omega(x_0,x_1,x_2) > E(x_0,x_0)+E_\omega(x_1,x_1)$. If $x_0 \not= x_2$ then $x_0, x_1, x_2$ are three distinct points.
Thus, $x_0<x_1$ implies $x_2<x_1$ and $x_1<x_0$ implies $x_1<x_2$. In both cases, lemma~\ref{lemma:twistProperty} tells us that
\begin{gather*}
E_\omega(x_0,x_1)+E_\omega(x_1,x_2) > E_\omega(x_0,x_2)+E_\omega(x_1,x_1).
\end{gather*}

Let $(x_0,\ldots,x_{n+1})$ be a nonmonotone sequence. We have two cases: either $x_0\leq x_n$ or $x_0\geq x_n$.
We shall only give the proof for the case $x_0 \leq x_n$.

{\it Case $x_0=x_n$.} Then $(x_0,\ldots,x_n)$ is nonmonotone and by induction
\begin{align*}
E_\omega(x_0,\ldots,x_{n+1}) &> E_\omega(x_n,x_{n+1}) + \sum_{i=0}^{n-1} E_\omega(x_i,x_i) \\
&= E_\omega(x_0,x_{n+1}) + \sum_{i=1}^{n} E_\omega(x_i,x_i).
\end{align*}
The conclusion holds whether $x_{n+1}=x_0$ or not.

{\it Case $x_0 < x_n$.} Whether $(x_0,\ldots,x_n)$ is monotone or not, we may  choose  a subset of indices $\{i_0,\ldots,i_r\}$ such that $i_0=0$, $i_r=n$, $x_{i_0} < x_{i_1} < \ldots < x_{i_r}$ and
\[
E_\omega(x_0,\ldots,x_{n+1}) \geq \Big(E_\omega(x_{i_0},\ldots,x_{i_r}) + \sum_{i \not\in \{i_0,\ldots,i_r\}} E_\omega(x_i,x_i) \Big) + E_\omega(x_n,x_{n+1}).
\]

If $x_n\leq x_{n+1}$, then $(x_0,\ldots,x_n)$ is necessarily nonmonotone and the previous inequality is strict.  If $x_n=x_{n+1},$ the lemma is proved by modifying $i_r=n+1$. If $x_n < x_{n+1}$, the lemma is proved by choosing $r+1$ indices and  $i_{r+1}=n+1$.

If $x_{n+1}< x_n=x_{i_r}$, by applying lemma \ref{lemma:twistProperty}, one obtains
\begin{gather*}
E_\omega(x_{i_{r-1}},x_{i_r}) + E_\omega(x_n,x_{n+1}) > E_\omega(x_n,x_{i_r}) + E_\omega(x_{i_{r-1}},x_{n+1}), \\
E_\omega(x_0,\ldots,x_{n+1}) > E_\omega(x_{i_0},\ldots,x_{i_{r-1}},x_{n+1}) + \big[ \sum_{i \not\in \{i_0,\ldots,i_r\}} E_\omega(x_i,x_i) \big] + E_\omega(x_n,x_n).
\end{gather*}
\noindent If $x_{i_{r-1}} < x_{n+1}$, the lemma is proved by changing $i_r=n$ to $i_r = n+1$. If $x_{i_{r-1}}=x_{n+1}$, the lemma is proved by choosing $r-1$ indices and $i_{r-1} = n+1$. If $x_{n+1} < x_{i_{r-1}}$, we apply again lemma \ref{lemma:twistProperty} until there exists a largest $s\in\{0,\ldots,r\}$  such that $x_s < x_{n+1}$ or $x_{n+1} \leq x_0$. In the former case, the lemma is proved by choosing $s+1$ indices and by modifying $i_{s+1}=n+1$.
In the latter case, namely, when $x_{n+1} \leq x_0 < x_n$, we have
\[
E_\omega(x_0,\ldots,x_{n+1}) > E_\omega(x_0,x_{n+1}) + \sum_{i=1}^n E_\omega(x_i,x_i)
\]
\noindent and the lemma is proved whether $x_{n+1}=x_0$ or $x_{n+1}<x_0$.
\end{proof}

The Ma\~n\'e subadditive cocycle $\Phi(\omega,t)$ (definition \ref{definition:ManeSubCocycle2}) is obtained by minimizing a normalized energy $E_\omega(x_0,\ldots,x_n)- n\bar E$ on all the configurations satisfying $x_0=0$ and $x_n=t$. The following lemma shows that it is enough
to minimize on strictly monotone configurations (unless $t=0$).

\begin{corollary}
\label{corollary:twistMonotony}
If $L$ is weakly twist, then, for every $\omega\in\Omega$, the Ma\~n\'e subadditive cocycle $\Phi(\omega,t)$ satisfies:
\par-- if $t=0$, $\Phi( \omega,0)=E_\omega(0,0)-\bar E$,
\par-- if $t>0$, $\Phi(\omega,t)=\inf_{n\geq1} \inf_{0=x_0<x_1<\ldots<x_n=t} [E_\omega(x_0,\ldots,x_n)-n\bar E]$,
\par-- if $t<0$, $\Phi(\omega,t)=\inf_{n\geq1} \inf_{0=x_0>x_1>\ldots>x_n=t} [E_\omega(x_0,\ldots,x_n)-n\bar E]$.
\end{corollary}

\begin{proof}
Lemma \ref{lemma:twistMonotony} tells us that we can minimize the energy of  $E_\omega(x_0,\ldots,x_n)-n\bar E$ by the sum of two terms:

\noindent -- either $x_n=x_0$, then
\[
E_\omega(x_0,\ldots,x_n)-n\bar E  \geq \big [E_\omega(x_0,x_0)-\bar E\big] + \sum_{i\notin\{0,n\}} \big[E_\omega(x_i,x_i)-\bar E \big];
\]
-- or $x_n\not= x_0$, then for some $(x_{i_0},\ldots,x_{i_r})$  strictly monotone, with $ i_0 = 0 $ and $ i_r = n $,
\[
E_\omega(x_0,\ldots,x_n)-n\bar E  \geq \big[E_\omega(x_{i_0},\ldots,x_{i_r})-r\bar E\big] + \sum_{i\not\in\{i_0,\ldots,i_r\}} \big[E_\omega(x_i,x_i)-\bar E \big].
\]
We conclude the proof by noticing that $\bar E \leq \inf_{x\in\mathbb{R}}E_\omega(x,x)$.
\end{proof}

We recall that a finite configuration $(x_0,x_1,\ldots,x_n)$ is said to be minimizing in the environment $\omega$ if $E_\omega(x_0,x_1,\ldots,x_n) \leq E_\omega(y_0,y_1,\ldots,y_n)$ whenever $x_0=y_0$ and $x_n=y_n$. The following lemmas show that, under certain conditions, a minimizing configuration is
strictly monotone.

\begin{lemma}
\label{lemma:StrictlyMonotone}
 Suppose that $L$ is weakly twist. Then for every $\omega\in\Omega$, if $(x_0,\ldots,x_n)$ is a minimizing configuration for $E_\omega$,
 with $x_0\not= x_n$, such that $x_i$ is strictly between $x_0$ and $x_n$ for every $0<i<n-1$, then $(x_0,\ldots,x_n)$ is strictly monotone.
\end{lemma}

\begin{proof}
Let $(x_0,\ldots,x_n)$ be such a minimizing sequence. We show, in part 1, it is monotone, and, in part 2, it is strictly monotone.

{\it Part 1.} Assume by contradiction that  $(x_0,\ldots,x_n)$ is not monotone. According to lemma \ref{lemma:twistMonotony}, one can find a subset of indices $\{i_0,\ldots,i_r\}$ of $\{0,\ldots,n\}$, with $i_0=0$ and $i_r=n$, such that  $(x_{i_0},\ldots,x_{i_r})$ is strictly monotone and
\[
E_\omega(x_0,\ldots,x_n) > E_\omega(x_{i_0},\ldots,x_{i_r}) + \sum_{i\not\in \{i_0,\ldots,i_r\}} E_\omega(x_i,x_i).
\]
We choose  the largest integer $r$ with the above property. Since $(x_0,\ldots,x_n)$ is not monotone, we have necessarily $r<n$. Since $(x_0,\ldots,x_n)$ is minimizing, one can find  $i\not\in\{i_0,\ldots,i_r\}$ such that $x_i \not\in\{x_{i_0},\ldots,x_{i_r}\}$. Let $s$ be one of the indices of $\{0,\ldots,r\}$ such that $x_i$ is between $x_{i_s}$ and $x_{i_{s+1}}$. Then, by lemma~\ref{lemma:twistProperty},
\[
E_\omega(x_{i_s},x_{i_{s+1}})+E_\omega(x_i,x_i) > E_\omega(x_{i_s},x_i)+E_\omega(x_i,x_{i_{s+1}}).
\]
We have just contradicted the maximality of $r$. Therefore, $(x_0,\ldots,x_n)$ must be monotone.

{\it Part 2.} Assume by contradiction that $(x_0,\ldots,x_n)$ is not strictly monotone. Then $(x_0,\ldots,x_n)$ contains a subsequence of the form $(x_{i-1},x_i,\ldots,x_{i+r},x_{i+r+1})$ with $ r \ge 1 $ and $ x_{i-1}\not= x_i=\ldots=x_{i+r}\not= x_{i+r+1}$. To simplify the proof, we assume $x_{i-1}<x_{i+r+1}$. We want built a configuration $(x_{i-1}',x_i',\ldots,x_{i+r}',x_{i+r+1}')$ so that $ x_{i-1}'=x_{i-1} $,
$ x_{i+r+1}' = x_{i+r+1} $ and
$$ E_\omega(x_{i-1},x_i,\ldots,x_{i+r},x_{i+r+1}) > E_\omega(x_{i-1}',x_i',\ldots,x_{i+r}',x_{i+r+1}'). $$
Indeed, since $(x_{i-1},\ldots,x_{i+r+1})$ is minimizing, we have
\[
E_\omega(x_{i-1},\ldots,x_{i+r+1}) = E_\omega(x_{i-1},x_i+\epsilon,x_{i+1}-\epsilon,\ldots,x_{i+r}-\epsilon, x_{i+r+1}) + o(\epsilon^2).
\]
Let
$$ \alpha= \frac{1}{x_i-x_{i-1}}\int_{x_{i-1}}^{x_{i}}\frac{\partial^2 E_\omega}{\partial x\partial y}(x,x_i) \,dx < 0, $$
$$ \beta = \frac{1}{x_{i+r+1}-x_{i+r}}\int_{x_{i+r}}^{x_{i+r+1}} \frac{\partial^2 E_\omega}{\partial x\partial y}(x_{i+r},y) \,dy < 0. $$
By Aubry crossing lemma,
\begin{align*}
E_\omega(x_{i-1},\, &x_i+\epsilon) + E_\omega(x_i+\epsilon,x_{i+1}-\epsilon)  \\
&= E_\omega(x_{i-1},x_{i+1}-\epsilon) + E_\omega(x_i+\epsilon,x_i+\epsilon) -2\epsilon(x_i-x_{i-1})\alpha +o(\epsilon).
\end{align*}
Since $ x_i = x_{i+r} $, obviously $ E_\omega(x_i+\epsilon, x_i +\epsilon) = E_\omega(x_{i+r}+\epsilon, x_{i+r} +\epsilon) $.
Again by Aubry crossing lemma,
\begin{align*}
E_\omega & (x_{i +r}+\epsilon, x_{i +r} +\epsilon) + E_\omega(x_{i+r}-\epsilon, x_{i+r+1}) \\
&= E_\omega(x_{i+r}-\epsilon,x_{i+r}+\epsilon) + E_\omega(x_{i+r}+\epsilon,x_{i+r+1}) -2\epsilon(x_{i+r+1}-x_{i+r})\beta + o(\epsilon).
\end{align*}
Then, for $\epsilon$ small enough, we have
\begin{align*}
E_\omega(x_{i-1},\ldots,x_{i+r+1}) > E_\omega(x_{i-1},x_i-\epsilon,\ldots,x_{i-r-1}-\epsilon,x_{i+r}+\epsilon, x_{i+r+1}),
\end{align*}
which contradicts that $(x_{i-1},\ldots,x_{i+r+1})$ is minimizing.  We have thus proved that $(x_0,\ldots,x_n)$ is strictly monotone.
\end{proof}

\section{Backgrounds on quasicrystals and the notion of locally constant Lagrangians}\label{section:Frenkel-Kontorova-quasicrystals}

\subsection{One-dimensional quasicrystals}\label{sec:QCDelone} 
Our purpose in this section is to  provide  a rich variety of examples of almost crystalline interaction models. We first recall the basic definitions and properties concerning quasicrystals. More details on such a motivating concept can be found, for instance, in \cite{BellissardBenedettiGambaudo,KellendonkPutnam,LagariasPleasants}. Associated with quasicrystals,
we will consider strongly equivariant functions (an inspiration to our concept of locally transversally constant Lagrangian to be introduced in section~\ref{sec:hull}). We recall their main properties here and we refer the reader to
\cite{GambaudoGuiraudPetite2006_01,Kellendonk2003} for the proofs.

\paragraph{Definition of a quasicrystal.}
For a discrete set $\omega \subset \mathbb{R}$, a $\rho$-{\em patch}, or a {\em pattern} for short,  is a finite set  $\texttt{P}$ of the form $\omega \cap \overline{B_{\rho}(x)}$ for some $x\in \omega$ and some constant $\rho>0$, where $B_{\rho}(x)$ denotes the open ball of radius $\rho$ centered in $x$. We say that $y \in \omega$ is an {\em occurrence} of $\texttt{P}$ if  $\omega \cap \overline{B_{\rho}(y)}$ is equal to \texttt{P} up to a translation.
A {\em quasicrystal} is a discrete set $\omega \subset \mathbb{R}$ satisfying
\begin{itemize}
\item[--] {\em finite local complexity}: for any $\rho>0$,  $\omega$ has just a finite number of $\rho$-patches up to translations;
\item[--] {\em repetitivity}: for all $\rho > 0$, there exists $M(\rho) > 0$ such that any closed ball of
radius $M(\rho)$ contains at least one occurrence of every $\rho$-patch of $\omega$;
\item[--] {\em uniform pattern distribution}: for any pattern ${\texttt P}$ of $\omega$, uniformly in $x \in \mathbb R$, the following positive limit exists
$$\lim_{r \to +\infty}\frac{\# \left(\{y \in \mathbb R : y \textrm{ is an occurrence of } {\texttt P}\}\cap B_r(x)\right)}{\text{Leb} (B_r(x))}
= \nu({\texttt P}) > 0.$$
\end{itemize}

Notice that the finite local complexity is equivalent to the fact that the intersection of the difference set $\omega - \omega$ with any bounded set is finite. Basic examples of one-dimensional quasicrystals are the lattice $\mathbb Z$ and the Beatty sequences defined by
$ \omega(\alpha) = \{n \in \mathbb Z : \lfloor n\alpha \rfloor - \lfloor (n-1) \alpha \rfloor \} $ for $ \alpha \in (0,1) $.
Moreover, when $\alpha$ is irrational as in example~\ref{ex:3}, the set $\omega(\alpha)$ provides a non periodic quasicrystal 
for which the repetitively and the uniform pattern distribution are due to the minimality and the unique ergodicity  
of an irrational rotation on the circle. For details, we refer to~\cite{LagariasPleasants}. 

Note that, from the definition, when $\omega$ is a quasicrystal, then the discrete set $\omega + t$, 
obtained by translating any point of $\omega$ by $t \in \mathbb R$, is also a quasicrystal.
A quasicrystal is said to be {\em aperiodic} if $\omega + t = \omega$ implies $t=0$, and {\em periodic} otherwise. 
For Beatty sequences, it is simple to check that the quasicrystal $\omega(\alpha)$ is aperiodic if, and only if, 
$\alpha$ is irrational.

\paragraph{Hull of a quasicrystal.}
Given a quasicrystal $\omega_{*} \subset \mathbb R $, we will equip the set $\omega_{*} + \mathbb R$ of all the translations of $\omega_{*}$ with a topology that reflects its combinatorial properties:  the Gromov-Hausdorff topology.  Roughly speaking, two quasicrystals  in this set will be close whenever they have the same pattern in a large neighborhood of the origin, up to a small translation.

Such a topology is metrizable  and an associated metric can be defined as follows (for details, see \cite{BellissardBenedettiGambaudo,Kellendonk2003}):
given ${\omega}$ and $\underline\omega$ two translations of $\omega_{*}$, their distance is
\begin{align*}
D( \omega, \underline\omega)  := \inf \big\{ \frac{1}{r+1} :
\exists \, \vert t \vert, \vert {\underline t} \vert< \frac{1}{r} \, \textrm{ s.t. }
(\omega + t)\cap B_{r}(0)=({\underline\omega} + {\underline t})\cap B_{r}(0) \big\}.
\end{align*}
The \emph{continuous hull} $ \Omega(\omega_{*}) $ of the quasicrystal $\omega_{*}$ is the completion of this metric space. The finite local complexity hypothesis implies that $\Omega(\omega_{*}) $ is a compact metric space and that any element $\omega \in \Omega(\omega_{*})$ is a quasicrystal which has the same patterns as $\omega_{*}$ up to translations (see \cite{KellendonkPutnam, BellissardBenedettiGambaudo}). Moreover, $\Omega (\omega_{*})$ is equipped with a continuous $\mathbb R$-action given by the homeomorphisms
$$\tau_{t} \colon \omega \mapsto \omega- t \quad \mbox{ for }  \omega \in \Omega(\omega_{*}).$$

The dynamical system $(\Omega(\omega_{*}), \{\tau_{t}\}_{t\in \mathbb R})$ has a dense orbit, namely, the orbit of~$\omega_{*}$. Actually, the repetitivity hypothesis is equivalent to the {\em minimality} of the action, and so any orbit is dense. The uniform pattern distribution is equivalent to the {\em unique ergodicity}: the $\mathbb R$-action has a unique invariant probability measure. For details on these properties, we refer the reader to \cite{KellendonkPutnam, BellissardBenedettiGambaudo}. We summarize these facts in the following proposition.

\begin{proposition}[\cite{KellendonkPutnam, BellissardBenedettiGambaudo}] Let $\omega_{*}$ be a quasicrystal of $\mathbb R$. Then the dynamical system $(\Omega(\omega_{*}), \{\tau_{t}\}_{t\in \mathbb R})$ is minimal and uniquely ergodic.
\end{proposition}

\paragraph{Flow boxes.}\label{sec:flowboxdecompositionQC} 
The \emph{canonical transversal} $ \Xi_{0}(\omega_{*}) $ of the hull $\Omega (\omega_{*})$ of a quasicrystal is the set of quasicrystals $\omega$ in $\Omega(\omega_{*})$ such that the origin $0$ belongs to $\omega$. A basis of the topology on $\Xi_0(\omega_*)$ is given by  cylinder sets $\Xi_{\omega,\rho}$ with $\omega \in \Xi_0(\omega_*)$ and $\rho>0$. 
In general, that is,  for every  $\omega \in \Omega(\omega_{*})$ and  $\rho>0$ such that $\omega\cap B_{\rho}(0) \neq \emptyset$,
 a {\it cylinder set} $\Xi_{\omega,\rho}$ is defined  by
$$\Xi_{\omega, \rho} := \{\underline{{\omega}} \in \Omega(\omega_{*}) :
\omega \cap \overline{B_{\rho}(0)} = {\underline{\omega}} \cap \overline{B_{\rho}(0)} \}.$$
If $\omega \in\Xi_0(\omega_*)$, then $\Xi_{\omega, \rho}  \subset \Xi_0(\omega_*)$.

The designation of transversal comes from the obvious fact that the set $\Xi_{0}(\omega_{*})$ is transverse to the action: for any real  $t$ small enough, we have  $\tau_{t}(\omega) \not\in \Xi_{0}(\omega_{*})$
for any $\omega \in \Xi_{0}(\omega_{*})$. This gives a Poincar\'e section.

\begin{proposition}[\cite{KellendonkPutnam}]
The canonical transversal $\Xi_{0}(\omega_{*})$ and the cylinder sets $\Xi_{\omega, \rho}$ associated with an aperiodic quasicrystal  $\omega_{*}$ are Cantor sets. If $\omega_{*}$ is a periodic quasicrystal, these sets are finite.
\end{proposition} 

This allows us to give a more dynamical description of the hull in one dimension 
by considering the {\em return time} function $\Theta : \Xi_{0}(\omega_{*}) \to \mathbb R^+$ defined by
$$ \Theta(\omega) := \inf \{t >0 : \tau_{t} ({\omega}) \in \Xi_{0}(\omega_{*}) \}, \quad \forall \, \omega \in \Xi_{0}(\omega_{*}). $$
The finite local complexity implies that this function is locally constant.
The {\em first return map} $T \colon \Xi_{0}(\omega_{*}) \to \Xi_{0} (\omega_{*})$ is then given by
$$ T (\omega) := \tau_{\Theta(\omega)}(\omega), \quad \forall \, \omega \in  \Xi_{0}(\omega_{*}).$$
Remark that the unique invariant probability measure on $\Omega(\omega_{*})$ induces a finite measure on $\Xi_{0}(\omega_{*})$ that is $T$-invariant (see \cite{GambaudoGuiraudPetite2006_01}).

It is straightforward to check that the dynamical system $(\Omega(\omega_{*}), \{\tau_{t}\}_{t\in \mathbb R})$ is conjugate to the suspension of the map $T$ on the set $\Xi_{0}(\omega_{*})$ with the time map given by the function $\Theta$.
Thus, when $\omega_{*}$ is periodic, the hull $\Omega(\omega_{*})$ is homeomorphic to a circle. Otherwise, $\Omega(\omega_{*})$ has a laminated structure: it is locally the Cartesian product of a Cantor set by an interval. 

To be more precise, in the aperiodic case, for every $\omega\in\Omega(\omega_*)$ and $r>0$, if $\rho$ is large enough, the set
$$  U_{\omega, \rho, r} := \{ {\underline{\omega}} - t : t \in B_r(0),\  \underline{\omega} \in \Xi_{{\omega},\rho}\} $$
is  open  and homeomorphic to $ B_r(0)\times \Xi_{\omega,\rho}$ by the map
$(t,\underline{\omega}) \to \tau_t(\underline{\omega}) =  \underline{\omega}-t$. Their collection forms a base for the topology of 
$\Omega(\omega_{*})$.  In this case, $ U_{\omega, \rho, r}$ is called {\it a flow box} of the cylinder set $\Xi_{\omega,\rho}$.

The next lemma improves the fact that the return time is locally constant.

\begin{lemma}[\cite{BellissardBenedettiGambaudo}]\label{lem:flowboxDelone}
Let $\omega_{*}$ be an aperiodic quasicrystal.  Let $U_i := U_{\omega_i, \rho_i,r_i}$, $i=1,2$, be two flow boxes such that $U_1\cap U_2 \not=\emptyset$. Then there exists a real number $a\in\mathbb{R}$ such that, for every $\underline{\omega}_i \in \Xi_{\omega_i,\rho_i}$, 
$|t_i|<r_i$, $i=1,2$,
\[
\underline{\omega}_1-t_1 = \underline{\omega}_2-t_2  \quad\Longrightarrow\quad t_2 = t_1 - a.
\]
\end{lemma}

\paragraph{Strongly equivariant function.}\label{paragraph:StrgEquiv} 
Associated with a quasicrystal  $\omega_{*}$ of $\mathbb R$, we will consider strongly $\omega_{*}$-equivariant functions, as introduced in \cite{Kellendonk2003}. A potential  $V_{\omega_{*}} \colon \mathbb{R} \to \mathbb{R}$ is said to be {\em strongly} $\omega_{*}$-{\em equivariant} (with range $R$) if there exists a constant $R>0$ such that   
$$ V_{\omega_{*}}(x) = V_{\omega_{*}}(y), \qquad \forall \ x, y \in \mathbb R \, \text{ with } 
\left(B_{R}(x) \cap \omega_{*} \right)  - x = \left(B_{R}(y) \cap \omega_{*} \right) - y. $$

Of course any periodic potential  is strongly equivariant with respect to a discrete  lattice of periods.
In example \ref{ex:3}, the function $V_{\omega(\alpha)}$ is strongly $\omega(\alpha)$-equivariant with range 
$R = \lfloor \frac 1\alpha \rfloor +1$.  Let us mention another example from \cite{Kellendonk2003}, which holds for any quasicrystal $\omega_{*}$. Let $\delta := \sum_{x \in \omega_{*}} \delta_x $ be the Dirac comb supported on the points of a quasicrystal $\omega_{*}$ and let $ g \colon \mathbb R \to \mathbb R$ be a smooth function with compact support. Then, one may check that the convolution product $\delta * g$ is a smooth strongly $\omega_{*}$-equivariant function. Actually, any strongly $\omega$-equivariant function can be defined by a similar procedure \cite{Kellendonk2003}.

A strongly equivariant potential factorizes through a continuous function on the hull $\Omega(\omega_{*})$. More precisely,
the following lemma shows that  strongly $\omega_{*}$-equivariant functions arise from functions on the space $\Omega(\omega_{*})$ that are constant on the cylinder sets.

\begin{lemma}[\cite{GambaudoGuiraudPetite2006_01,Kellendonk2003}]\label{lem:equivfunct} Given a quasicrystal $\omega_{*}$ of  $\mathbb R$, let $V_{\omega_{*}} \colon \mathbb R \to \mathbb R$ be a continuous strongly  $\omega_{*}$-equivariant function with range $R$. Then, there exists a unique continuous function $V \colon \Omega(\omega_{*}) \to \mathbb R$ such that
$$ V_{\omega_{*}}(x) = V \circ \tau_{x}(\omega_{*}), \quad \forall \, x \in \mathbb R.$$
Moreover,  $V$ is constant on any cylinder set $\Xi_{\omega, R+S}$, with $\omega \in \Omega(\omega_{*})$ and $S\geq0$.
In addition, if $ V_{\omega_{*}} $ is $ C^2 $, then $ V $ is $ C^2 $ along the flow (that is, for all $ \omega $, the function
$ x \in \mathbb R \mapsto V(\tau_x(\omega)) $ is $ C^2 $).
\end{lemma}

Note that, for every $ S  > 0$, 
the function $V \colon \Omega(\omega_{*}) \to \mathbb R$ is {\it transversally constant} on each flow box 
$ U_{\omega, R+S,  S}$,  that is,
$$ V\left(\tau_x(\underline{\omega})\right) = V\left(\tau_x(\underline{\omega}')\right), \qquad \forall \, |x| < S, \
\forall \, \underline{\omega}, \underline{\omega}' \in \Xi_{\omega, R+S}. $$
This comes from the fact that  $\tau_{x}(\underline{\omega}') \in \Xi_{\tau_{x}(\underline{\omega}), R}$ whenever  
$\underline{\omega}, \underline{\omega}' \in \Xi_{\omega, R+S}$ and $ |x| < S$, since $V$ is constant on such cylinder sets. 

\subsection{Flow boxes and locally constant Lagrangians}\label{sec:hull}

In order to complete the definition of almost crystalline interaction models  (definition~\ref{definition:quasicrystal_WeaklyTwistProperty}), 
we introduce here the technical concept of a locally transversally constant Lagrangian that we adopt in this paper. By doing this, 
we focus on a class of models whose typical examples are provided by suspensions of minimal and uniquely ergodic homeomorphisms on 
a Cantor set, with locally constant ceiling functions. Such a modeling approach enables us to consider general $\mathbb R$-actions, as, for instance, equicontinuous, distal or expansive  ones, whereas, in the aperiodic quasicrystal case, one deals always with expansive actions. 
We also show that strongly equivariant functions associated with a quasicrystal provide locally transversally constant Lagrangians.

In topological dynamics, the study of minimal homeomorphisms on a Cantor set has been enriched by an invaluable combinatorial 
description of the system via Kakutani-Rohlin towers (see, for instance, \cite{GPS}). Using a similar strategy in our context, we describe, in a second part, the transverse measures associated with the probability measures on the space $\Omega$ invariant by the flow $\tau$. Characterized by the average frequency of return times to a particular transverse section of the flow, these measures are key ingredients in the proof of theorem~\ref{theorem:MainContribution}.

\paragraph{Precise notion of locally constant Lagrangians.}
Our definition of a locally transversally constant Lagrangian is based on (topological) flow boxes, transverse sections, and flow box decompositions. Even if we consider only the one-dimensional case, these concepts can be introduced in any dimension.

\begin{definition}
\label{definition:BoxDecomposition}
Let $(\Omega,\{\tau_t\}_{t\in\mathbb{R}})$ be an almost periodic environment.

\noindent -- An open  set $U\subset\Omega$ is said to be a flow box of size $R>0$ if there exists a compact subset $\Xi\subset\Omega$, called transverse section, such that:

\noindent \quad  $\centerdot$ the induced topology on $\Xi$ admits a basis of closed and open subsets, called clopen subsets,

\noindent \quad $\centerdot$ $ \tau(t,\omega) = \tau_t(\omega)$, $ (t, \omega) \in \mathbb R \times \Xi $, is a homeomorphism from $B_R(0)\times\Xi$ onto $U$.

\noindent We shall later write $B_R=B_R(0)$   and $\tau_{(i)}^{-1} = \tau^{-1}_{\vert U_{i}} :U_{i}\to B_R\times\Xi$ for a flow box $U_{i}$.

\noindent  -- Two flow boxes $U_i=\tau[B_{R_i}\times\Xi_i]$ and $U_j = \tau[B_{R_j}\times\Xi_j]$ are said to be admissible if,  whenever $U_i\cap U_j\not=\emptyset$, there exists $a_{i,j} \in \mathbb{R}$ such that
\[
\tau_{(j)}^{-1}\circ \tau(t,\omega) = (t-a_{i,j},\tau_{a_{i,j}}(\omega)), \quad \forall \, (t,\omega) \in  \tau_{(i)}^{-1}(U_i \cap U_j).
\]

\noindent -- A flow box decomposition $\{U_i\}_{i\in I}$ is a cover of $\Omega$ by admissible flow boxes.
\end{definition}

 Lemma \ref{lem:flowboxDelone} implies that the hull of a quasicrystal admits a flow box decomposition given by flow boxes of cylinder sets \cite{BellissardBenedettiGambaudo}.
Standard examplifications of the structures formalized in definition~\ref{definition:BoxDecomposition} 
are provided by the suspensions of minimal homeomorphisms on Cantor sets, 
with locally constant ceiling functions. This context includes expansive flows (as in the case of one-dimensional quasicrystals) and equicontinuous ones. 
But, in general, a minimal flow does not possess a cover of flow boxes. 

An interaction model does not have a canonical notion of vertical section. 
Such a notion occurs naturally whenever the model admits a flow box decomposition. 
More importantly, in this situation, we give and exploit a definition 
of locally transversally constant Lagrangian.

\begin{definition}
\label{definition:TransversallyConstant}
Let $(\Omega,\{\tau_t\}_{t\in\mathbb{R}},L)$ be an almost periodic interaction model admitting a flow box decomposition.

\noindent -- A flow box $\tau[B_R \times \Xi]$ is said to be compatible with respect to a flow box decomposition $\{U_i\}_{i\in I}$,
where  $U_i=\tau[B_{R_i} \times \Xi_i]$,  if  for every $|t| < R$, there exist $i\in I$,  $|t_i| < R_i$ and a  clopen subset $\tilde\Xi_i$ of $\Xi_i$  such that $\tau_t(\Xi) = \tau_{t_i}(\tilde\Xi_i)$.

\noindent -- $L$ is said to be {\em locally transversally constant}  with respect to a flow box decomposition $\{U_i\}_{i \in I}$ if, for every flow box $\tau[B_R \times \Xi]$  compatible with respect to $\{U_i\}_{i \in I}$,
\[
\forall \, \omega,\omega'\in\Xi, \ \forall \, |x|,|y|<R, \quad E_{\omega'}(x,y) = E_\omega(x,y).
\]
\end{definition}

As in examples~\ref{ex:1} and \ref{ex:3},
interaction models with weakly twist and locally transversally constant Lagrangians can be easily built when the interaction energy has the form
$E_\omega(x,y) = W(y-x)+V_1(\tau_x(\omega))+V_2(\tau_y(\omega))$, where $W$ is {\em superlinear weakly convex} (namely, $W$ is $C^2$, $W''>0$ a.e.  and $|W'(t)|\to+\infty$ as $|t|\to+\infty$), and $ V_1 $ and $ V_2 $ are locally transversally constant, in the sense described below.

\begin{definition}
\label{definition:LocallyTransversallyConstant}
Let $(\Omega,\{\tau_t\}_{t\in\mathbb{R}},L)$ be an almost periodic interaction model.
A function $V:\Omega\to\mathbb{R}$ is said to be locally transversally constant with respect to a 
flow box decomposition $\{U_i\}_{i\in I}$, where $U_i=\tau(B_{R_i}\times \Xi_i)$, if
\[
\forall \, i\in I, \ \forall \, \omega,\omega'\in\Xi_i, \ \forall \, |x|< R_i, \quad V(\tau_x(\omega)) = V(\tau_x(\omega')).
\]
\end{definition}

Notice that, in example~\ref{ex:2}, the locally transversally constant property does not hold.
We check in the next lemma that locally transversally constant functions  $V_1, V_2:\Omega\to\mathbb{R}$ indeed enable to construct a transversally constant Lagrangian. 

\begin{lemma}
\label{lemma:LocallyTransversallyConstant}
Let $(\Omega,\{\tau_t\}_{t\in\mathbb{R}},L)$ be an  almost periodic interaction model admitting a flow box decomposition. Let $V_1,V_2:\Omega\to\mathbb{R}$ be two locally transversally constant functions on the same flow box decomposition, and $W=\mathbb{R}\to\mathbb{R}$ be any function. Define $L(\omega,t)=W(t)+V_1(\omega)+V_2(\tau_t(\omega))$. Then $L$ is locally transversally constant.
\end{lemma}

\begin{proof}
Assume $V_1$ and $V_2$ are locally transversally constant on a flow box decomposition $\{U_i\}_{i \in I}$. Let $\tau[B_R \times \Xi]$ be a flow box which is compatible with respect to $\{U_i\}_{i \in I}$.  If $|x|,|y|<R$ and $\omega,\omega'\in \Xi$, then
\begin{gather*}
E_{\omega}(x,y) = W(y-x)+V_{1,\omega}(x)+V_{2,\omega}(y).
\end{gather*}
There exist $i \in I$, $|t_i|<R_i$ and $\tilde\Xi_i$ a clopen subset of $\Xi_i$ such that $\tau_x(\Xi) = \tau_{t_i}(\tilde\Xi_i)$. Then $\tau_x(\omega)=\tau_{t_i}(\omega_i)$ and $\tau_x(\omega')=\tau_{t_i}(\omega_i')$ for some $\omega_i,\omega'_i \in \tilde\Xi_i$. We have
\begin{gather*}
V_{1,\omega}(x) = V_{1, \omega_i}(t_i) = V_{1, \omega'_i}(t_i) = V_{1, \omega'}(x).
\end{gather*}
Similarly $V_{2, \omega}(y) = V_{2, \omega'}(y)$. We have thus proved $E_{\omega'}(x,y) = E_\omega(x,y)$.
\end{proof}

To give a concrete example of a family of locally transversally constant Lagrangians for which the conclusions of Theorem~\ref{theorem:MainContribution} hold, let us recall that  a continuous function $V \colon \Omega \to \mathbb{R}$ is $C^{2}$ {\em along the flow} if, for each $\omega \in \Omega$, the function $x \in \mathbb{R} \mapsto V(\tau_{x}(\omega))$ is $C^{2}$.

\begin{corollary}
\label{corollary:MainContribution}
Let $(\Omega,\{\tau_t\}_{t\in\mathbb{R}})$ be an almost periodic  environment admitting a flow box decomposition. Let $V_1, V_2:\Omega\to\mathbb{R}$ be $C^0$
locally transversally constant functions (on the same flow box decomposition) that are $ C^2 $ along the flow. Let $W \colon \mathbb{R}\to\mathbb{R}$ be a
$C^2$ superlinear weakly convex function. Define
\[
L(\omega,t) = W(t)+V_1(\omega) +V_2(\tau_t(\omega)).
\]
Then $L$ is $C^0$,  superlinear, weakly twist   and  locally transversally constant. If moreover $(\Omega,\{\tau_t\}_{t\in\mathbb{R}})$ is uniquely ergodic, then  $(\Omega,\{\tau_t\}_{t\in\mathbb{R}},L)$ is an almost crystalline  interaction model  and all conclusions of theorem~\ref{theorem:MainContribution} apply.
\end{corollary}

\paragraph{Kakutani-Rohlin tower description of the system.} 
Flow boxes are open sets obtained by taking the union of every orbits of size $R$ 
starting from any point belonging to a closed transverse Poincar\'e section.
The restricted topology on a transverse section must be special: it must admit a basis of clopen sets. 
We recall in lemma~\ref{lemma:KakutaniRohlin} how to construct a suspension with locally constant return maps  called Kakutani-Rohlin tower.
When  the flow is uniquely ergodic, we describe in the lemmas \ref{lemma:TransverseMeasure} and \ref{lemma:TransverseMeasure2} how  this Kakutani-Rohlin tower enables to characterize the  unique transverse measure associated with each transverse section.

We begin with some basic properties of systems with a flow box decomposition. 
Since the proof of the next lemma is standard, we leave it to the reader.

\begin{lemma}
\label{lemma:FlowBoxes}
Let $(\Omega,\{\tau_t\}_{t\in\mathbb{R}})$ be an almost periodic environment. Assume that the action is not periodic ($t\in\mathbb{R}\mapsto \tau_t(\omega) \in \Omega$ is injective for every $\omega\in\Omega$). Then
\begin{enumerate}
\item If $\tau[B_R\times\Xi]$ is a flow box, then there exists $R'$ such that
\[
\Omega=\tau[B_{R'}\times\Xi] = \{\tau_t(\omega) : |t|<R' \ \text{and} \ \omega\in\Xi\}.
\]
\item If $\tau[B_R\times\Xi]$ is a flow box, then $\tau:\mathbb{R}\times\Xi \to \Omega$ is open and $\tau[B_R\times\Xi']$ is again a flow box for every clopen subset $\Xi'\subset\Xi$.
\item If $\tau[B_R\times\Xi]$ is a flow box, then, for every $R'>0$ and $\omega\in\Xi$, there exists a clopen set $\Xi'\subset\Xi$ containing $\omega$ such that $\tau[B_{R'}\times\Xi']$ is again a flow box.
\item If $U=\tau[B_R\times\Xi]$ and $U'=\tau[B_{R'}\times \Xi']$ are two admissible flow boxes, if $\tau[B_{2R+2R'} \times \Xi]$ and $\tau[B_{2R+2R'} \times \Xi']$ are also flow boxes, then
\[
U\cap U' = \tau(\tilde B \times \tilde\Xi) = \tau(\tilde B' \times \tilde \Xi')
\]
for some clopen sets $\tilde \Xi$, $\tilde \Xi'$ and some open convex subsets $\tilde B \subset B_R$,  $\tilde B' \subset B_{R'}$.
\item If $\{U_i\}_{i \in I}$  is a flow box decomposition, then, for every $\omega \in \Omega$ and $R>0$, there exits a flow box $\tau [B_R \times \Xi]$, with a transverse section $\Xi$ containing $\omega$, that is compatible with respect to $\{U_i\}_{i \in I}$.
\end{enumerate}
\end{lemma}

The existence of a flow box decomposition enables us to build a global transverse section of the flow with locally constant return times. We extend for an almost periodic interaction model what has been done for quasicrystals in \cite{GambaudoGuiraudPetite2006_01}.  We first define the notion of Kakutani-Rohlin tower and show that an interaction model possessing a flow box decomposition admits a Kakutani-Rohlin tower.

\begin{definition}
\label{definition:KakutaniRohlin}
Let $(\Omega,\{\tau_t\}_{t\in\mathbb{R}})$ be a one-dimensional almost periodic environment possessing a flow box decomposition $\{U_i\}_{i \in I}$. We call Kakutani-Rohlin tower a partition $\{F_\alpha\}_{\alpha \in A}$ of $\Omega$ of the form
$$F_\alpha = \tau\big([0,H_\alpha) \times \Sigma_\alpha\big) = \cup_{0\leq t<H_\alpha}\tau_t(\Sigma_\alpha)$$
for some some height $H_\alpha>0$ and some transverse section $\Sigma_\alpha$ (closed set admitting a basis of clopen subsets),
where $\tau\big( (0,H_\alpha)\times \Sigma_\alpha\big)$ is a flow box (open and homeomorphic to $ (0,H_\alpha) \times \Sigma_\alpha$), and
$ \cup_{\alpha \in A} \tau(\{H_\alpha\}\times\Sigma_\alpha) = \cup_{ \alpha \in A} \tau(\{0\}\times\Sigma_\alpha) = \cup_{\alpha \in A} \Sigma_\alpha$.
Moreover, we say that a Kakutani-Rohlin tower is compatible with respect to $\{U_i\}_{i \in I}$ if, for every $\alpha \in A$,  there exist $i\in I$, $t_i\in\mathbb{R}$ and a clopen subset $\tilde\Xi_i\subset \Xi_i$ such that $\Sigma_\alpha = \tau_{t_i}(\tilde\Xi_i)$ and
$ [t_i,t_i+H_\alpha) \subset [-R_i,R_i) $.
\end{definition}

\begin{lemma}
\label{lemma:KakutaniRohlin}
Let $(\Omega,\{\tau_t\}_{t\in\mathbb{R}})$ be a one-dimensional almost periodic environment  possessing a flow box decomposition $\{U_i\}_{i \in I}$. Then there exists  a Kakutani-Rohlin tower $\{F_\alpha\}_{\alpha \in A}$ which is compatible with respect to $\{U_i\}_{i \in I}$.
\end{lemma}

\begin{proof}
Let $\{U_i\}_{i=1}^n$ be a flow box decomposition, where $U_i = \tau[B_{R_i} \times \Xi_i]$. By definition, $U_i$ is an open set of $\Omega$.
We denote $V_i := \tau \big([-R_i, R_i) \times \Xi_i\big)$. We shall build by induction on $i = 1, \ldots, n $ a collection of flow boxes
$\{\tau\big((0,H_{i,j})\times \Sigma_{i,j}\big)\}_j$ such that

-- the sets $F_{i,j} := \tau\big([0,H_{i,j})\times\Sigma_{i,j}\big)$ are pairwise disjoint,

-- $ V_i \setminus \cup_{k<i} V_k = \cup_{j} \tau\big([0,H_{i,j})\times\Sigma_{i,j}\big) = \cup_j F_{i,j} $,

-- $ \tau(\{-R_i\} \times \Xi_i) \setminus \cup_{k<i}V_k  \subset \cup_j \tau(\{0\} \times \Sigma_{i,j})$,

-- $\cup_{k<i} \tau(\{R_k\} \times \Xi_k)  \cap (V_i \setminus \cup_{k<i}V_k) \subset \cup_j \tau(\{0\} \times \Sigma_{i,j})$,

-- $\tau(\{H_{i,j}\}\times\Sigma_{i,j})\cap  \cup_{k<i} V_k  \subset \cup_{k < i} \cup_j \tau(\{0\} \times \Sigma_{k,j})$,

-- $\tau(\{H_{i,j}\}\times\Sigma_{i,j})\setminus \cup_{k<i} V_k  \subset \tau(\{R_i\} \times \Xi_i) \setminus \cup_{k<i} V_k $.

\noindent For $i=1$, we choose $H_{1,1}=2R_1$ and $\Sigma_{1,1}=\tau_{-R_1}(\Xi_1)$.  Assume  that we have built the sets
$\tau\big([0,H_{k,j})\times\Sigma_{k,j}\big)$
for every $k<i$ and $j$. Thanks to the admissibility of the flow boxes $\{U_i\}_{i\in I}$, the set  $V_i \cap V_k$, if  nonempty,  is of the form $\tau(J_{i,k}\times \Xi_{i,k})$, where $J_{i,k}=[a_{i,k},b_{i,k})$, with $-R_i \leq a_{i,k} < b_{i,k} \leq R_i$, and $\Xi_{i,k}$ is a clopen set of $\Xi_i$.
The complement $V_i \setminus V_k$ is the union of sets of the form
\[
\tau\big([-R_i,a_{i,k}) \times \Xi_{i,k}\big), \quad \tau\big([b_{i,k},R_i) \times \Xi_{i,k}\big) \quad\text{or}\quad
\tau\big([-R_i,R_i) \times (\Xi_i \setminus \Xi_{i,k})\big).
\]
Hence, $V_i \setminus \cup_{k<i} V_k$ is obtained as a disjoint union of sets $\tau\big([c_{\alpha},d_{\alpha}) \times \tilde\Sigma_{\alpha}\big)$,
where $\tilde\Sigma_{\alpha}$ is any clopen set of the form $\cap_{k<i} S_k$, with either $S_k =\Xi_{i,k}$ or $ S_k = \Xi_i \setminus \Xi_{i,k}$,
and $[c_{\alpha},d_{\alpha})$ corresponds to any connected component of $[-R_i,R_i) \, \setminus \cup_{k<i} J_{i,k}$. We next rewrite
$ \tau\big([c_{\alpha},d_{\alpha}) \times \tilde\Sigma_{\alpha}\big)$ as $ \tau\big([0,H_{i, j}) \times \Sigma_{i,j}\big)$, with $ j = j(\alpha) $,
where $\Sigma_{i, j} = \tau_{c_{\alpha}}(\tilde\Sigma_{\alpha})$ and  $H_{i,j} = d_{\alpha}- c_{\alpha}$. By construction, for all $k<i$
with $ V_i \cap \overline{V_k} \neq \emptyset $,  $ \tau(\{R_k\} \times \Xi_k)  \cap V_i  = \tau(\{b_{i,k}\} \times  \Xi_{i,k})$ and its part which is not in $\cup_{l<i}V_l$ is included into $ \cup_j \tau(\{0\} \times \Sigma_{i,j})$. Furthermore, $\tau(\{H_{i,j}\} \times \Sigma_{i,j})$ either is included into
$\tau(\{R_i\} \times \Xi_i)$ or intersects $V_k$ for some $k<i$ and therefore is included into $ \cup_{k < i} \cup_j \tau(\{0\} \times \Sigma_{k,j}) $.
\end{proof}

When a Kakutani-Rohlin tower is built, we obtain a global transverse section $\cup_{\alpha \in A}\Sigma_\alpha$ with a return time  constant on each $\Sigma_\alpha$ and equal to $H_\alpha$. We can induce on a particular section $\Sigma_{\alpha_0}$ and build a second Kakutani-Rohlin tower with larger heights. We explain in the next paragraph the notations that will be used for these successive towers.

If $\{F^0_\alpha\}_{\alpha \in A^0}$ is a Kakutani-Rohlin tower of order 0, denote  $F^0_\alpha := \tau\big([0,H^0_\alpha) \times \Sigma^0_\alpha\big)$. We say that $\Sigma^0 := \cup_\alpha \Sigma^0_\alpha$ is the basis of the tower. Let $\omega_*$ be a reference point of the base $\Sigma^0$. Consider $\alpha_0$ such that $\omega_* \in \Sigma^0_{\alpha_0}$. The construction of the tower of order~1 is done by inducing the flow on $\Sigma^1 := \Sigma^0_{\alpha_0}$. We obtain a  partition of $\Sigma^1$ given by $\{\Sigma^1_\beta\}_{\beta \in A^1}$, where $\beta = (\alpha_0,\ldots,\alpha_p)$, $p\geq 1$, $\alpha_p = \alpha_0$,
$\alpha_i \not = \alpha_0$ for $i=1,\ldots,p-1$,
\begin{equation*}
\Sigma^1_\beta = \Sigma^0_{\alpha_0} \cap \tau_{H^0_{\alpha_0}}^{-1}(\Sigma^0_{\alpha_1})\cap \ldots \cap \tau^{-1}_{H^0_{\alpha_0}+\ldots+ H^0_{\alpha_{p-1}}}(\Sigma^0_{\alpha_p}).
\end{equation*}
By minimality, there is a finite collection of such nonempty sets $\Sigma^1_\beta$. Define then
\begin{gather}
H_\beta^1 := H^0_{\alpha_0} + \ldots + H^0_{\alpha_{p-1}},
\notag \\
F^1_\beta := \tau\big([0,H^1_\beta) \times \Sigma^1_\beta\big)= \bigcup_{i=0}^{p-1} \tau\big([t_i,t_i + H^0_{\alpha_i}) \times \Sigma^0_{\alpha_i}\big),  \,\,
\text{with } t_i= \sum_{j=0}^{i-1} H^0_{\alpha_j}.
\label{equation:InducedTower}
\end{gather}
We have just obtained a new Kakutani-Rohlin tower  $\{F^1_\beta\}_{\beta \in A^1}$ of basis $\Sigma^0_{\alpha_0}$. We induced again on the section $\Sigma^1_{\beta_0}$ that contains $\omega_*$ and build the tower of order~2. We shall write $\{F^l_\alpha\}_{\alpha \in A^l}$ for the successive towers that are built using this procedure and $F^l_*$ for the tower of height $H^l_*$ whose basis $\Sigma^l_*$ contains $\omega_*$ . The preceding construction gives $\min_{\alpha \in A^{l+1}} H^{l+1}_\alpha \geq H^l_{*}$ and in particular $H^{l+1}_{*} \geq H^l_*$. It may happen that $H^l_*=H^{l+1}_* = H^{l+2}_* = \ldots$ In that case, the flow is a suspension over $\Sigma^l_*$ of constant return time $H^l_*$ (and $\Omega$ is isomorphic to $\Sigma^l_* \times S^1$). In order to exclude this situation, we split the basis $\Sigma^l_{\alpha_0}$ which contains $\omega_*$ into two disjoint clopen sets $\Sigma^l_{\alpha_0} = \Sigma^l_{\alpha'_0}  \cup \Sigma^l_{\alpha''_0}$. We obtain again a K
 akutani-Rohlin tower and we induce as before on the subset which contains $\omega_*$. If $(\Omega,\{\tau_t\}_{t\in\mathbb{R}})$ is not periodic, we may choose the splitting so that $H^{l+1}_{*} > H^l_*$ at each step of the construction.

We now assume that the flow $(\Omega, \{\tau_t\}_{t\in\mathbb{R}})$ is uniquely ergodic. Let $\lambda$ be the unique ergodic invariant probability measure. The average frequency of return times to a transverse section of a flow box measures the thickness of the section. The next lemma gives a precise definition of a family of transverse measures $\{\nu_\Xi\}_{\Xi}$ parameterized by every transverse section $\Xi$.

\begin{lemma}
\label{lemma:TransverseMeasure}
Let $(\Omega,\{\tau_t\}_{t\in\mathbb{R}})$ be an almost periodic and uniquely ergodic  environment.
Given $\Xi$ a transverse section, let $\mathcal{R}_\Xi(\omega)$ be the set of return times to $\Xi$,
\[
\mathcal{R}_\Xi(\omega) := \{ t\in\mathbb{R} : \tau_t(\omega) \in \Xi \}, \quad \forall \, \omega \in \Omega.
\]
Then, for every nonempty clopen set $\Xi' \subset \Xi$, the following limit exists uniformly with respect to $\omega \in \Omega$ and is positive:
\[
\nu_\Xi(\Xi') := \lim_{T\to+\infty} \frac{\#(\mathcal{R}_{\Xi'}(\omega)\cap B_T(0))}{\text{\rm Leb}(B_T(0))} > 0.
\]
Moreover, $\nu_\Xi$ extends to a finite and nonnegative measure on $\Xi$, called transverse measure to $\Xi$, and,
for every flow box $U = \tau[B_R\times\Xi]$,
\[
\lambda(\tau(B'\times \Xi'))=\text{\rm Leb}(B') \nu_\Xi(\Xi'), \quad \forall \, B'\subset B_R(0), \ \forall \, \Xi'\subset \Xi \quad (\text{Borel sets}).
\]
\end{lemma}

\begin{proof}
Let $U=\tau[B_R\times\Xi]$ be a flow box. Let  $t_1\not=t_2$ be two return times of $\mathcal{R}_{\Xi}(\omega)$. Since $\tau$ is injective on $B_R(0)\times\Xi$,
it is straightforward that $B_R(t_1) \cap B_R(t_2) = \emptyset$. For $\omega\in\Omega$ and $ T > 0 $, consider
\[
\mu_{T,\omega}(U') = \frac{1}{\text{\rm Leb}(B_T(0))} \int_{B_T(0)} \mathbf{1}_{U'}(\tau_s(\omega)) \,ds, \quad \forall \, U' \subset\Omega
\quad (\text{Borel set}).
\]
The unique ergodicity of the action implies that, for all $ \phi \in C^0(\Omega) $, $\mu_{T,\omega}(\phi)$ converges uniformly
in $\omega$ to $\lambda(\phi)$ as $T\to+\infty$. Let $B'\subset B_R(0)$ be a Borel set and $\Xi' \subset \Xi$ be a nonempty clopen set.
For $U' = \tau(B' \times \Xi')$, notice then that
\begin{equation*}
\{ s \in \mathbb{R} : \tau_s(\omega) \in U' \} = \bigcup_{t\in\mathcal{R}_{\Xi'}(\omega)} t+B', \,\,\,
\mu_{T,\omega}(U') = \sum_{t\in\mathcal{R}_{\Xi'}(\omega)} \frac{\text{\rm Leb}(B_T(0)\cap (t+B'))}{\text{\rm Leb}(B_T(0))},
\end{equation*}
and, whenever $ T > 2R $,
\begin{equation*}
\text{\rm Leb}(B') \frac{\#(B_{T-R}(0)\cap \mathcal{R}_{\Xi'}(\omega))}{\text{\rm Leb}(B_T(0))} \leq \mu_{T,\omega}(U') \leq \text{\rm Leb}(B') \frac{\#(B_{T+R}(0)\cap \mathcal{R}_{\Xi'}(\omega))}{\text{\rm Leb}(B_T(0))}.
\end{equation*}
Moreover, clearly $\#(B_T(0)\cap \mathcal{R}_{\Xi'}(\omega)) \leq \frac{\text{\rm Leb}(B_{T+R}(0))}{\text{\rm Leb}(B_R(0))} $ and
$ \lim_{T\to+\infty} \frac{\text{\rm Leb}(B_{T+R}(0))}{\text{\rm Leb}(B_T(0))} = 1 $. Thus, if $B'$ is open in $B_R(0)$,
then $U'$ is open in $\Omega$ and
\[
\lambda(U') \leq \liminf_{T\to+\infty} \mu_{T,\omega}(U') \leq \frac{\text{\rm Leb}(B')}{\text{\rm Leb}(B_{2R}(0))}.
\]
In particular, if $B'$ is negligible, thanks to the regularity of $\text{\rm Leb}$, $\lambda(U')=0$. If $B'$ is open,
$\overline{B'} \subset B_R(0)$ and $\partial B'$ is negligible, then, for every $\epsilon>0$,  there exist nonnegative continuous functions
$\phi\leq \psi$ such that
\[
\phi \leq \mathbf{1}_{\tau(B'\times\Xi)} \leq \mathbf{1}_{\tau(\overline{B'}\times\Xi)} \leq \psi \quad\text{and}\quad \lambda(\psi-\phi)<\epsilon.
\]
Therefore, $\mu_{T,\omega}(\tau(B'\times\Xi'))$ converges uniformly in $\omega$ to $\lambda(\tau(B'\times\Xi))$ as $T\to+\infty$.

On the one hand, for all clopen set $\Xi'\subset\Xi$, $\tau(B_R(0)\times \Xi')$ is a flow box and
\[
\lim_{T\to+\infty} \frac{\#(B_{T}(0)\cap \mathcal{R}_{\Xi'}(\omega))}{\text{\rm Leb}(B_T(0))} := \nu_\Xi(\Xi') \quad(\text{exists uniformly in $\omega$}).
\]
On the other hand, for every $B'=B_{R'}(s')$, $s'\in B_R(0)$, $\|s'\|+R' < R$,
\[
\lambda(\tau(B'\times\Xi')) = \lim_{T\to+\infty} \mu_{T,\omega}(\tau(B'\times\Xi')) = \text{\rm Leb}(B')\nu_\Xi(\Xi').
\]
Hence, $\nu_\Xi$ extends to a measure on the Borel sets of $\Xi$ and by the monotone class theorem
$\lambda(\tau(B'\times\Xi')) = \text{\rm Leb}(B')\nu_\Xi(\Xi')$ for every Borel sets $B' \subset B_R(0)$ and $\Xi' \subset \Xi$.

We finally remark that $\nu_\Xi(\Xi')>0$ for every nonempty clopen set $\Xi'\subset \Xi$,
since otherwise there would exist an open set of $ \Omega $ of $\lambda$-measure zero.
\end{proof}

We come back to Kakutani-Rohlin towers of flows. Let $\{F^l_\alpha\}_{\alpha \in A^l}$ be such a tower of order $l$ and $\{F^{l+1}_\beta\}_{\beta \in A^{l+1}}$ be the subsequent tower as introduced in  (\ref{equation:InducedTower}). We recall the definition of the homology matrix as explained in lemma 2.7 of \cite{GambaudoGuiraudPetite2006_01}. For every $\alpha \in A^l$ and $\beta \in A^{l+1}$, $\beta=(\alpha_0,\ldots,\alpha_p)$, $\alpha_0=\alpha_p$, $ \alpha_i \not = \alpha_0 $ for $ i = 1, \ldots, p - 1$, we denote
\begin{equation*}
M^l_{\alpha,\beta} := \# \{ 0 \leq k \leq p-1 : \alpha_k = \alpha \}.
\end{equation*}
A flow box of order $l+1$, $\tau\big([0,H^{l+1}_\beta)  \times \Sigma^{l+1}_\beta\big)$, is obtained as a disjoint union of flow boxes of order
$l$ of the type $\tau\big([t_i, t_i + H^l_{\alpha_i}) \times \Sigma^l_{\alpha_i}\big)$. The integer $M^l_{\alpha,\beta}$ counts the number of times
a flow box of order $l+1$ indexed by $\beta$ cuts a flow box of order $l$ indexed by $\alpha$. The main result that we shall need is given by the following lemma.

\begin{lemma}\label{lemma:TransverseMeasure2}
Let $(\Omega,\{\tau_t\}_{t\in\mathbb{R}})$ be a one-dimensional almost periodic and uniquely ergodic  environment. Let $\{F^l_\alpha\}_{\alpha \in A^l}$ be a sequence of Kakutani-Rohlin towers built as in~(\ref{equation:InducedTower}).  Let $\nu^l$ be the transverse measure associated with the transverse section $\cup_{\alpha \in A^l}\Sigma^l_\alpha$. If $\nu^l_\alpha := \nu^l(\Sigma^l_\alpha)$, then
\[
\nu^l_\alpha = \sum_{\beta \in A^{l+1}} M^l_{\alpha,\beta} \nu^{l+1}_\beta.
\]
\end{lemma}

\begin{proof}
Let $\Xi = \cup_{\beta \in A^{l+1}} \Sigma^{l+1}_\beta$. For $\omega \in \Xi$, let $ 0 = t_0, t_1, t_2, \ldots$ be its successive return times to $\Xi$.
We introduce as in lemma~\ref{lemma:TransverseMeasure} the set of return times to the transverse section $\Sigma^l_\alpha$, say,
$\mathcal{R}^l_\alpha(\omega) := \{t \in \mathbb{R} : \tau_t(\omega) \in \Sigma^l_\alpha \}$. The set $\mathcal{R}^{l+1}_\beta(\omega)$ is defined similarly. Since
\[
\# \big(\mathcal{R}^l_\alpha(\omega) \cap [0,t_n) \big) = \sum_{\beta \in A^{l+1}}  M^l_{\alpha,\beta} \ \# \big( \mathcal{R}^{l+1}_\beta(\omega) \cap[0,t_n) \big),
\]
we divide by $t_n$ and apply lemma \ref{lemma:TransverseMeasure} to conclude.
\end{proof}

\section{Calibrated configurations for transversally constant Lagrangians}\label{sec:conf-cal}

This section is devoted to the proof  of the second main result of this paper: theorem~\ref{theorem:MainContribution}.
We consider an almost periodic environnement $(\Omega, \{\tau_{t}\}_{t\in \mathbb R})$ admitting a flow box decomposition with respect to which  $ L: \Omega \times \mathbb R \to \mathbb R $ is locally transversally constant, and we suppose the Lagrangian $ L $ is also weakly 
twist. We shall study the properties of the associated minimizing configurations. 

If $E_\omega(x,x) = \bar E$ for some $\omega \in \Omega$ and $x \in \mathbb{R}$, then $\delta_{(\tau_x(\omega),0)} \in \mathbb{M}_{min}(L)$, $\tau_x(\omega)$ belongs to the projected Mather set, and the configuration $x_{k,\omega}=x$ fulfills items \ref{item:MainContribution_1} and \ref{item:MainContribution_2} of  theorem~\ref{theorem:MainContribution}.
We thus suppose later $E_\omega(x,x) > \bar E$ for every $\omega $ and $x$.

Our first nontrivial result is stated in proposition~\ref{proposition:BoundedJumps}: a finite configuration $(x^n_0,\ldots,x^n_n)$ which realizes the minimum of the energy among all configurations of the same length must be strictly monotone, and must have uniformly bounded jumps $|x^n_k - x^n_{k-1}| \leq R$. Our second key result, proposition \ref{proposition:RotationNumber}, shows actually  that $\liminf_{n\to+\infty} \frac{1}{n} |x^n_n - x^n_0| > 0$: the frequency of points $x^n_k$ in a flow box of sufficiently large size is positive. We finally conclude this section with the proof  of theorem~\ref{theorem:MainContribution}.

\begin{lemma}
\label{lemma:StrictlyMonotoneBis}
Given a weakly twist and transversally constant Lagrangian $L$, there exists $R>0$ such that, if $\omega \in \Omega$ is any environment, if  $(x_0,\ldots,x_n) \in\mathbb{R}$ is  minimizing  for $E_\omega$ and  $|x_n-x_0|\geq R$, then $(x_0,\ldots,x_n)$ is strictly monotone.
\end{lemma}

\begin{proof}
Let $\{U_i = \tau[B_{R_i} \times \Xi_i] \}_{i\in I}$ be a flow box decomposition with respect to which $L$ is transversally constant. Since $ \{U_i\}_{i \in I} $ is a finite cover, we may choose $R$ large enough so that every orbit of size $R$ meets every box entirely: for every $\omega$, for every $|y-x|\geq R$, for every $i\in I$, there exists $t_i\in\mathbb{R}$ such that $(t_i - R_i, t_i + R_i) \subset [x,y]$ and $\tau_{t_i}(\omega)\in\Xi_i$.

We first show that there cannot exist $r\geq0$  and $0<k<n-r$ such that
\[
 x_k < x_{k-1}, \quad x_k=\ldots=x_{k+r} \quad\text{and}\quad x_k < x_{k+r+1}.
\]
Otherwise, Aubry crossing lemma implies that
\[
E_\omega(x_{k-1},x_k)+E_\omega(x_k,x_{k+r+1}) > E_\omega(x_{k-1},x_{k+r+1})+E_\omega(x_k,x_k).
\]
We rewrite the configuration $(x_0,\ldots,x_{k-1},x_{k+r+1},\ldots,x_n)$ as $(y_0,\ldots,y_{n-r-1})$.
Let $U_{i}$ be a flow box containing $\tau_{x_k}(\omega)$. There exists $|s|<R_i$  and $\omega' \in\Xi_i$ such that $\tau_{x_k}(\omega) = \tau_{s}(\omega')$. By the choice of $R$, there exists $t$ such that $(t-R_i,t+R_i)\subset  [x_0,x_n]$ and $\tau_{t}(\omega)\in\Xi_i$. Let $z_0= \ldots = z_r  := t+s$ and $1 \leq l \leq n-r-1$ be such that $y_{l-1}<  z_0 \leq y_l$. Using the fact that $L$ is transversally constant on $U_{i}$, we have
\begin{gather*}
E_\omega(x_k,x_k) = E_{\omega'}(s,s) = E_{\tau_{t}(\omega)}(s,s)=E_\omega(z_0,z_0).
\end{gather*}
By applying again Aubry crossing lemma, we obtain
\[
E_\omega(y_{l-1},y_l) + E_\omega(z_0,z_0) \geq E_\omega(y_{l-1},z_0) + E_\omega(z_0,y_l),
\]
(possibly with a strict inequality if $ z_0<y_l$). We have just obtained a new configuration $(y_0,\ldots,y_{l-1},z_0,\ldots, z_r,y_l,\ldots,y_{n-r-1})$ of $n$ points with a strictly lower energy, which contradicts the fact that $(x_0,\ldots,x_n)$ is minimizing.

There cannot exist similarly $r\geq0$  and $0<k<n-r$ such that
\[
 x_k > x_{k-1}, \quad x_k=\ldots=x_{k+r} \quad\text{and}\quad x_k > x_{k+r+1}.
\]
There cannot exist either a sub-configuration $(x_{k-1},x_k,\ldots,x_{k+r},x_{k+r+1})$, $ r \ge 1 $, of the form $x_{k-1}\not= x_{k+r+1}$ and $x_k=\ldots =x_{k+r}$ strictly between $x_{k-1}$ and $x_{k+r+1}$ thanks to lemma~\ref{lemma:StrictlyMonotone}. We are thus left to a configuration of the form
\[
x_0=\!\ldots\!=x_{r}< \!\ldots\! <  x_{n-r'}=\!\ldots\!=x_n \quad \text{or} \quad x_0=\!\ldots\!=x_{r} > \!\ldots\! >  x_{n-r'}=\!\ldots\!=x_n
\]
for some $r, r'\geq0$. Assume by contradiction that $x_0=x_1$ (the case $x_{n-1}=x_n$ is done similarly). Exactly as before, there exist
$U_{i}$ containing $\tau_{x_0}(\omega)$, $|s|<R_i$ and $\omega'\in\Xi_i$ such that $\tau_{x_0}(\omega)=\tau_{s}(\omega')$, as well as
there exists $t \in\mathbb{R}$ such that $(t - R_i, t + R_i) \subset  [\min\{x_0, x_n\}, \max\{x_0, x_n\}]$ and $\tau_{t}(\omega)\in\Xi_i$.
One can show in an analogous way that, whenever $z :=t+s $ belongs to $ (\min\{x_{l-1}, x_l\}, \max\{x_{l-1}, x_l\}] $ for $ 2 \le l \le n $,
$E(x_0, x_1, \ldots, x_n) \ge E(x_1,\ldots,x_{l-1},z,x_l,\ldots,x_n)$, with strict inequality if $ z < \max\{x_{l-1}, x_l\} $. Since
$ (x_0, x_1, \ldots, x_n) $ is a minimizing configuration, this implies that $ z = \max\{x_{l-1}, x_l\}  \not \in \{x_0,x_n\}$, and  $(x_1,\ldots,x_{l-1},z,x_l,\ldots,x_n)$
is a minimizing configuration.
The first part of this proof shows that this cannot happen.

The proof that $(x_0,\ldots,x_n)$ is strictly monotone is complete.
\end{proof}

\begin{proposition}
\label{proposition:BoundedJumps}
Given a weakly twist and transversally constant Lagrangian $L$, there exists $R>0$ such that, 
if $\omega\in\Omega$ is any environment and $(x_0,\ldots,x_n)$, $n\geq2$, 
satisfies $ E(x_0, \ldots, x_n) = \min_{(y_0, \ldots, y_n)} E_\omega(y_0,\ldots,y_n)$ and
$\max_{0\leq k < l \leq n}|x_k-x_l|\geq R$,  then $(x_0,\ldots,x_n)$ is strictly monotone and $\sup_{1 \leq k \leq n}|x_k - x_{k-1}| \leq R$.
\end{proposition}

\begin{proof}
Consider $\omega\in\Omega$, $n\geq2$, and $(x_0,\ldots,x_n)$ realizing the minimum of the energy among all configurations of length
$n$ in the environment $\omega$.

{\it Part 1.} We show there exists $R'>0$  (independent from $\omega$ and $n$) such that $|x_1-x_0| \leq R'$ and $|x_2-x_1| \leq R'$.
Indeed, we have
\begin{equation*}
E_\omega(x_0,x_1) \leq E_\omega(x_1,x_1) \quad\text{and}\quad E_\omega(x_0,x_1,x_2) \leq E_\omega(x_2,x_2,x_2),
\end{equation*}
which implies
\begin{equation*}
E_\omega(x_0,x_1) \leq \sup_{x\in\mathbb{R}} E_\omega(x,x) \quad\text{and}\quad E_\omega(x_1,x_2) \leq 2\sup_{x\in\mathbb{R}} E_\omega(x,x)-\inf_{x,y\in\mathbb{R}} E_\omega(x,y).
\end{equation*}
The existence of $R'$ follows then from the coerciveness of $L$, which is uniform with respect to $\omega$. Similarly,
we have $|x_{n-1} - x_{n-2}| \leq R'$ and $|x_n-x_{n-1}| \leq R'$.

{\it Part 2.} We show there exists $R''>0$ such that, if $(x_0,\ldots,x_m)$ is strictly monotone, then $|x_i-x_{i-1}| \leq R''$ for every $1\leq i\leq m$.
It is clear from the definition that, if $ L $ is transversally constant with respect to a particular flow box decomposition
$ \{\tau[B_{r_i} \times \Xi_i]\} $,
then $ L $  is transversally constant for any flow box decomposition such that its flow boxes are compatible with respect to $ \{\tau[B_{r_i} \times \Xi_i]\} $.
Therefore, let $\{U_i=\tau[B_{R'}\times  \Xi'_i]\}$ be a finite cover of $\Omega$ by flow boxes such that $\tau[B_{2R'}\times\Xi'_i]$ is again a flow box and
$L$ is transversally constant with respect to $\{\tau[B_{2R'}\times\Xi'_i]\}$. We choose $R''>0$ large enough so that every orbit of length $R''$ meets entirely each $\tau[B_{2R'}\times\Xi'_i]$. Let $U_i$ be a flow box containing $\tau_{x_1}(\omega)$: there exist $|s_1|<R'$ and $\omega'\in\Xi'_i$ such that $\tau_{x_1}(\omega)=\tau_{s_1}(\omega')$. From part 1, we deduce that $\tau[B_{2R'}\times\Xi'_i]$ contains $\{\tau_{x_0}(\omega),\tau_{x_1}(\omega),\tau_{x_2}(\omega)\}$. Denote  $s_0:=s_1+x_0-x_1$ and $s_2:=s_1+x_2-x_1$, so that 
$|s_0|,|s_2|<2R'$, $\tau_{x_0}(\omega)=\tau_{s_0}(\omega')$ and $\tau_{x_2}(\omega)=\tau_{s_2}(\omega')$. 
Assume by contradiction $|x_i-x_{i-1}|>R''$. Then, there exists $t\in\mathbb{R}$ such that
$(t-2R',t+2R') \subset [\min\{x_{i-1},x_i\}, \max\{x_{i-1},x_i\}]$ and $\tau_t(\omega)\in \Xi'_i$. Let $z_0=t+s_0$, $z_1=t+s_1$ and $z_2=t+s_2$. Notice that $(x_{i-1},x_i)$ and $(z_0,z_1,z_2)$ are ordered in the same way. As $L$ is transversally constant on $\tau[B_{2R'}\times\Xi'_i]$, we obtain
\[
E_\omega(x_0,x_1,x_2) = E_{\omega'}(s_0,s_1,s_2) = E_{\tau_t(\omega)}(s_0,s_1,s_2) = E_\omega(z_0,z_1,z_2).
\]
Aubry crossing lemma applied twice gives
\begin{align*}
E_\omega(x_{i-1},x_i) + E_\omega(z_0,z_1,z_2) &> E_\omega(x_{i-1},z_1) + E_\omega(z_0,x_i) + E_\omega(z_1,z_2), \\
&>  E_\omega(x_{i-1},z_1,x_i)  + E_\omega(z_0,z_2).
\end{align*}
As $L$ is transversally constant, $E_\omega(z_0,z_2)=E_\omega(x_0,x_2)$ as above and we obtain
\[
E_\omega(x_{i-1},x_i) + E_\omega(x_0,x_1,x_2) > E_\omega(x_{i-1},z_1,x_i)  + E_\omega(x_0,x_2).
\]
The configuration $(x_0,x_2,\ldots,x_{i-1},z_1,x_i,\ldots,x_m)$ has a strictly lower energy,
which contradicts the fact that $(x_0,\ldots,x_m)$ is minimizing. We obtain similarly that,
if $(x_{m},\ldots,x_n)$ is strictly monotone, then $|x_{i-1}-x_i| \leq R''$ for every $m+1 \leq i \leq n$.

{\it Part 3.} Let  $R'''$ be the constant given by lemma \ref{lemma:StrictlyMonotoneBis}. Take $R>2R''+4R'''$. If $|x_n-x_0|>R'''$,
then $(x_0,\ldots,x_n)$ is strictly monotone by lemma \ref{lemma:StrictlyMonotoneBis} and the jumps $|x_i-x_{i-1}|$ are uniformly bounded by $R''$. The proof is finished.

Assume by contradiction that  $|x_n-x_0| \leq R'''$. Let $a=\min_{0\leq k \leq n}x_k$ and $b= \max_{0\leq k \leq n}x_k$. Since $\text{\rm diam}(\{x_k:0\leq k \leq n\}) \geq R$, one of the two inequalities $|a-x_0|>R/2$ or $|b-x_0|>R/2$ must be satisfied. Assume to simplify $|b-x_0|>R/2$ (the case $|a-x_0|>R/2$ is done similarly). Hence, $b=x_m$ for some $0< m < n$. Since $(x_0,\ldots,x_m)$ and $(x_m,\ldots,x_n)$ are minimizing and satisfy $|x_m-x_0|>R'''$ and $|x_m-x_n|>R'''$, these two configurations are strictly monotone. Then, part~2 tells us that the jumps $|x_i-x_{i-1}|$ are uniformly bounded by $R''$. In particular,
$|x_{m+1}-x_m| \leq R''$. The configuration $(x_0,\ldots,x_{m+1})$ is minimizing and, since $|x_m-x_0|>R''+2R'''$, it satisfies $|x_{m+1}-x_0|>R'''$. By lemma \ref{lemma:StrictlyMonotoneBis}, it must be strictly monotone, which is in contradiction with the maximum $x_m$. 

Thus, $ |x_n - x_0| > R''' $, $ (x_0, \ldots, x_n) $ is strictly monotone and  $|x_i-x_{i-1}| \leq R''$.
\end{proof}

The proof of the fact that $|x_k-x_{k-1}|$ is uniformly bounded uses the same ideas as in lemma 3.1 of \cite{GambaudoGuiraudPetite2006_01}.
The fact that $L$ is transversally constant enables us to translate subconfigurations without modifying the total energy.
For a minimizing and strictly monotone configuration, by minimality of the energy, two consecutive points cannot enclose a translated
subconfiguration of three points. More precisely, we have the following lemma that extends lemma 3.2 of  \cite{GambaudoGuiraudPetite2006_01}.

\begin{lemma}
\label{lemma:EquiDistribution}
Let $L$ be a weakly twist Lagrangian which is transversally constant for a flow box decomposition $\{U_i\}_{i \in I}$. Suppose that the flow box $\tau[B_R\times\Xi]$ is compatible with respect to $\{U_i\}_{i \in I}$. Let $(x_0,\ldots,x_n)$ be a strictly monotone minimizing configuration for some environment $\omega \in \Omega$. Let $(a-R,a+R)$ and $(b-R,b+R)$ be two disjoint intervals such that $\tau_a(\omega) \in \Xi$ and $\tau_b(\omega) \in \Xi$. Assume that $(a-R,a+R)$ is a subset of $[x_0,x_n]$. Let $A$ be the number of sites $0 \leq k \leq n$ such that $x_k $ belongs to $ (a-R,a+R)$  and let
$B$ be defined similarly. Then $ B \leq A+ 2$. In particular, if  $(b-R,b+R) \ \subset [x_0,x_n]$, then $|A-B| \leq 2$.
\end{lemma}

\begin{proof}
To simplify we assume that $(x_0,\ldots,x_n)$ is strictly increasing. The proof is done by contradiction by  assuming $B \geq A+3$. Denote
\begin{gather*}
\{y_1,\ldots,y_A\} := \{x_0,\ldots,x_n\}\cap (a-R,a+R) \quad\text{and}\\
\{y'_1,\ldots,y'_B\} := \{x_0,\ldots,x_n\}\cap (b-R,b+R).
\end{gather*}
Let $y_0$ be the greatest $x_k \leq a-R$ and $y_{A+1}$ be the smallest $x_k \geq a+R$.
We write $s_k := y'_k - b $ and $z_k := a+s_k$ for $k=1,\ldots,B$.
The partition into $A+1$ disjoint  intervals  $\cup_{k=1}^{A+1} \, (y_{k-1},y_k]$
must contain $ A+3 $ distinct points $\{z_1,\ldots,z_{A+3}\}$. We have therefore to consider two cases.

{\it Case 1.} Either some interval $ (y_{k-1},y_k] $, $2\leq k \leq A$, contains three points $(z_{i-1},z_i,z_{i+1})$.  By Aubry crossing lemma,
\begin{align*}
E_\omega(y_{k-1},y_k) + E_\omega(z_{i-1},z_i) &> E_\omega(y_{k-1},z_i) + E_\omega(z_{i-1},y_k), \\
E_\omega(z_{i-1},y_k) + E_\omega(z_i,z_{i+1}) &\geq E_\omega(z_{i-1},z_{i+1}) + E_\omega(z_{i}, y_k).
\end{align*}
Since $L$ is transversally constant on $\tau[B_R\times \Xi]$, we obtain
\begin{align*}
E_\omega(y'_{i-1},y'_i,y'_{i+1}) + E_\omega(y_{k-1},y_k) &= E_\omega(z_{i-1},z_i,z_{i+1}) + E_\omega(y_{k-1},y_k) \\
&> E_\omega(z_{i-1},z_{i+1}) + E_\omega(y_{k-1},z_i,y_k) \\
&= E_\omega(y'_{i-1},y'_{i+1}) +  E_\omega(y_{k-1},z_i,y_k).
\end{align*}
We have obtained a configuration (if, for instance, $b<a$) of the form 

\[
(x_0,\ldots,y'_{i-1},y'_{i+1},\ldots, y'_B,\ldots,y_1,\ldots,y_{k-1},z_i,y_k,\ldots,x_n)
\] 
with strictly lower energy, which contradicts the fact that $(x_0,\ldots,x_n)$ is minimizing.

{\it Case 2.} Or there exist two distinct intervals $(y_{k-1},y_k]$ and $(y_{l-1},y_l]$, with $2\leq k<l \leq A$, that contain each two points
$(z_{i-1},z_i)$ and $(z_{j-1},z_j)$, respectively. Notice that we may have $y_k = y_{l-1}$, but we must have $z_i < z_{j-1}$,
$ z_{i+1} \in \, (a-R,a+R) $, and possibly $z_{i+1}=z_{j-1}$. We want to obtain a contradiction by showing that one can decrease the sum of energies
$ E_\omega(y'_{i-1},\ldots,y'_j) + E_\omega(y_{k-1}, \ldots,y_l) $ while fixing the four boundary points.

In the case $z_i=y_k$, we perturb the point $ z_i $ slightly
by a small quantity $\epsilon$ and allow an increase of the energy of order $\epsilon^2$. Since $(z_{i-1},z_{i},z_{i+1})$  is minimizing, we have
\begin{gather*}
E_\omega(z_{i-1},z_{i},z_{i+1}) = E_\omega(z_{i-1},z_{i}-\epsilon,z_{i+1}) + o(\epsilon^2).
\end{gather*}
By Aubry crossing lemma, either $z_i < y_k$, and the reminder in lemma \ref{lemma:twistProperty} takes the form
\begin{gather*}
\text{reminder} := (z_{i-1}-y_{k-1})(z_i-y_k)\alpha > 0, \\
\text{where}\quad \alpha = \frac{1}{(z_{i-1}-y_{k-1})(z_i-y_k)} \int_{y_{k-1}}^{z_{i-1}} \int_{y_k}^{z_i} \frac{\partial^2   E_\omega}{\partial x\partial y} (x,y) \, dy dx < 0,
\end{gather*}
(in that case, we define $\epsilon :=0$), or $z_i=y_k$, and the reminder becomes
\begin{gather*}
\text{reminder} := -\epsilon(z_{i-1}-y_{k-1})\alpha + o(\epsilon) > o(\epsilon^2), \\
\text{where}\quad \alpha = \frac{1}{z_{i-1}-y_{k-1}} \int_{y_{k-1}}^{z_{i-1}}\! \frac{\partial^2  E_\omega}{\partial x\partial y} (x,y_k) \, dx <0.
\end{gather*}
In both cases, 
\begin{gather*}
E_\omega(y_{k-1},y_k) + E_\omega(z_{i-1},z_i-\epsilon)  = E_\omega(y_{k-1},z_i-\epsilon) + E_\omega(z_{i-1},y_k)  + \text{reminder}, \\ 
E_\omega(y_{k-1},y_k) + E_\omega(z_{i-1},z_{i},z_{i+1}) > E_\omega(y_{k-1},z_i-\epsilon, z_{i+1}) + E_\omega(z_{i-1},y_k).
\end{gather*}
Again by Aubry crossing lemma,
\[
E_\omega(y_{l-1},y_l) + E_\omega(z_{j-1},z_j) \geq E_\omega(y_{l-1},z_j) + E_\omega(z_{j-1},y_l),
\]
with possibly equality if $ z_j = y_l $. Since $L$ is transversally constant, we obtain
\begin{align*}
E_\omega(y'_{i-1},\ldots,y'_j)  &+ E_\omega(y_{k-1}, \ldots,y_l) \\
&= E_\omega(z_{i-1},\ldots,z_j)  + E_\omega(y_{k-1}, \ldots,y_l)\\
&> E_\omega(z_{i-1},y_k,\ldots,y_{l-1},z_j) + E_\omega(y_{k-1},z_i-\epsilon,z_{i+1},\ldots,z_{j-1},y_l) \\
&= E_\omega(y'_{i-1},w_k,\ldots,w_{l-1},y'_j) + E_\omega(y_{k-1},z_i-\epsilon,z_{i+1},\ldots,z_{j-1},y_l),
\end{align*}
with $t_k := y_k - a$, $w_k := b + t_k$,\ldots,$ t_{l-1}:= y_{l-1} - a $, $w_{l-1} := b+t_{l-1}$. Hence, we have a 
configuration $(\ldots,y'_{i-1},w_k,\ldots,w_{l-1},y'_j,\ldots,y_{k-1},z_i-\epsilon,z_{i+1},\ldots,z_{j-1},y_l,\ldots)$ 
with strictly lower energy,
which contradicts the fact that $(x_0,\ldots,x_n)$ is minimizing.
\end{proof}

It may happen that $E_\omega(x,x)=\bar E$ for some $\omega\in\Omega$ and $x \in\mathbb{R}$. Let $x_n=x$ for every $n$.
Then $(x_n)_{n\in\mathbb{Z}}$ is a calibrated configuration in the environment $\omega$ and $\delta_{(\tau_{x}(\omega),0)}$
is a minimizing measure. If $L$ is transversally constant on a flow box $\tau[B_R\times \Xi]$ such that $\tau_{x}(\omega)\in\Xi$,
then $\delta_{(\omega',0)}$ is a minimizing measure  for every $\omega'\in\Xi$. The projected Mather set contains $\Xi$ and  theorem \ref{theorem:MainContribution} is proved. We are thus left to understand the case
$\inf_{\omega\in\Omega, \ x\in\mathbb{R}}E_\omega(x,x) > \bar E$.

\begin{lemma}
\label{lemma:UnboundedConfiguration}
Let $L$ be a weakly twist Lagrangian for which
\[
\inf_{\omega\in\Omega, \ x\in\mathbb{R}} E_\omega(x,x) > \bar E.
\]
For every  $\omega\in\Omega$, $n\geq1$,  if $(x^n_0,\ldots,x^n_n)$ is  a configuration realizing the minimum  $E_\omega(x^n_0,\ldots,x^n_n) = \min_{x_0,\ldots,x_n\in\mathbb{R}} E_\omega(x_0,\ldots,x_n)$, then $\lim_{n\to+\infty} |x^n_n-x^n_0|=+\infty$.
\end{lemma}

\begin{proof}
The proof is done by contradiction. Let $\omega\in\Omega$ and $R>0$. Assume there exist infinitely many $n$'s for which every configuration $(x^n_0,\ldots,x^n_n)$ realizing the minimum of $E_\omega(x_0,\ldots,x_n)$ satisfies $|x^n_n-x^n_0|\leq R$. If $(x^n_0,\ldots,x^n_n)$ is not monotone,  thanks to lemma \ref{lemma:twistMonotony}, we  can find distinct indices $\{i_0,\ldots,i_r\}$ of $\{0,\ldots,n\}$ such that $i_0=0$, $i_r=n$, $(x^n_{i_0},\ldots,x^n_{i_r})$ is monotone (possibly not strictly monotone) and
\[
E_\omega(x^n_0,\ldots,x^n_n) \geq E_\omega(x^n_{i_0},\ldots,x^n_{i_r})+ \sum_{i\not\in\{i_0,\ldots,i_r\}} E_\omega(x^n_i,x^n_i).
\]
Let $\epsilon>0$ be chosen so that $E_\omega(x,y)\geq \bar E + \epsilon$ for every $|y-x|\leq \epsilon$.
Thus, if $ \theta_n $ denotes the number of indices $1\leq k \leq r$ such that $|x^n_{i_k}-x^n_{i_{k-1}}|>\epsilon$, it is clear that
$ \theta_n \le R/\epsilon $. Since
\[
n \bar E \geq E_\omega(x^n_0,\ldots,x^n_n) \geq (n - \theta_n)(\bar E+\epsilon) + \theta_n \inf_{x,y\in\mathbb{R}} E_\omega(x,y),
\]
we obtain a contradiction by letting $n\to+\infty$.
\end{proof}

We now assume that $(\Omega,\{\tau_t\}_{t \in \mathbb{R}},L)$ is an almost crystalline interaction model. We show in the following proposition that a sequence of configurations
$(x^n_0,\cdots,x^n_n)$ realizing the minimum of the energy $E_\omega(x_0,\ldots,x_n)$ among all configurations of length $n$
admits a weak rotation number in the sense that
\begin{equation}
\liminf_{n\to+\infty} \frac{|x^n_n - x^n_0|}{n} > 0.
\end{equation}
\noindent The existence of a rotation number for an infinite  minimizing configuration $(x_k)_{k\in\mathbb{Z}}$ has been established in \cite{GambaudoGuiraudPetite2006_01}. The following proposition extends partially  this result in two directions: the interaction model is more general; we compute the rotation number of a sequence of configurations of increasing length and not the rotation number of a unique infinite configuration.

\begin{proposition}
\label{proposition:RotationNumber}
Let $(\Omega,\{\tau_t\}_{t\in\mathbb{R}},L)$ be an almost crystalline interaction model. Assume that
\[
\inf_{\omega\in\Omega, \ x\in\mathbb{R}} E_\omega(x,x) > \bar E.
\]
For every $\omega\in \Omega$ and $n\geq1$, let $(x^{n}_0,\ldots,x^{n}_n)$ be a configuration realizing the minimum of the energy
among all configurations of length $n$:
\[
E_\omega(x^n_0,\cdots,x^n_n)=\min_{x_0,\ldots,x_n}E_\omega(x_0,\ldots,x_n).
\]
Then,

-- $\bar E = \lim_{n\to+\infty} \frac{1}{n} E_\omega(x^n_0,\cdots,x^n_n) = \sup_{n\geq1} \frac{1}{n} E_\omega(x^n_0,\cdots,x^n_n)$,

-- for $n$ sufficiently large, $(x^n_0,\cdots,x^n_n)$ is strictly monotone,

-- there is $R>0$ (independent of $ \omega $) such that $\sup_{n\geq1}\sup_{1\leq k \leq n}|x^n_k - x^n_{k-1}| \leq R$,

-- $\liminf_{n\to+\infty} \frac{1}{n} |x^n_n-x^n_0|  > 0$.
\end{proposition}

\begin{proof} To avoid trivialities, we assume that the flow  $(\Omega,\{\tau_t\}_{t\in\mathbb{R}})$ is not periodic.

{\it Step 1.} The first item has been proved in proposition~\ref{proposition:LocalGroundEnergy}; the limit can be obtained as a supremum
because of superadditivity.  Moreover, from lemma~\ref{lemma:UnboundedConfiguration}, $|x^n_n-x^n_0| \to +\infty$.
From proposition~\ref{proposition:BoundedJumps}, the configuration $(x^n_0,\ldots,x^n_n)$ must be strictly monotone and have
uniformly bounded jumps $R$. We are left to prove the last item of the proposition.

{\it Step 2.} By definition of an almost crystalline interaction model, 
$L$ is transversally constant with respect to some flow box decomposition $\{U_i\}_{i \in I}$ 
(definitions~\ref{definition:BoxDecomposition} and~\ref{definition:TransversallyConstant}). 
Let $\{F_\alpha\}_{\alpha \in A}$ be a Kakutani-Rohlin tower that is compatible with respect to
$\{U_i\}_{i \in I}$ (definition \ref{definition:KakutaniRohlin}) and let $\Sigma = \cup_{\alpha \in A} \Sigma_\alpha$ be its basis. 
We may assume that $\min_{\alpha \in A} H_\alpha$ is as large as we want and, in particular, larger than $R$ 
(see the construction (\ref{equation:InducedTower})).
We also assume that $n$ is sufficiently large so that every tower $F_\alpha$ of basis $\Sigma_\alpha$ is completely cut by the trajectory
$\tau_t(\omega)$ for $ t \in (\min\{x_0^n, x_n^n\}, \max\{x_0^n, x_n^n\})$. We consider $\nu$ the transverse measure to $\Sigma$ (as defined in
lemma~\ref{lemma:TransverseMeasure}) and we denote $\nu_\alpha := \nu(\Sigma_\alpha)$.

{\it Step 3.} Let $S^n<T^n$  be the two return times to $\Sigma$ (namely, $\tau_{S^n}(\omega) \in \Sigma$ and $\tau_{T^n}(\omega) \in \Sigma$) that are
chosen so that $[S^n,T^n)$ is the smallest interval containing the sequence $(x^n_k)_{k=0}^n$. From the definition of a Kakutani-Rohlin tower, $[S^n,T^n)$ can be written as a disjoint union of intervals of  type
$I_{\alpha,i} := [t_{\alpha,i},t_{\alpha,i}+H_\alpha)$, where the list $\{t_{\alpha,i}\}_{i}$, $i=1,\ldots,{C_\alpha^n}$,
denotes the successive return times to $\Sigma_\alpha$ between $S^n$ and $T^n$. We distinguish two exceptional intervals among this list: the two intervals which contain $x^n_0$ and $x^n_n$. If $x_0^n < x_n^n$, then $N_{\alpha,i}^n$ denotes the number of points $(x^n_k)_{k=1}^n$  belonging to $I_{\alpha,i}$ and $N^n_\alpha$ denotes the maximum of  $N^n_{\alpha,i}$.  If $x^n_n < x^n_0$, then $N^n_{\alpha,i}$ and $ N^n_\alpha $ are defined similarly by considering in
this case $(x^n_k)_{k=0}^{n-1}$. From lemma~\ref{lemma:EquiDistribution}, we obtain $N^n_\alpha -2 \leq N^n_{\alpha,i} \leq N^n_\alpha$ for every nonexceptional interval $I_{\alpha,i}$. We show that $\sup_{n\geq1}N^n_\alpha < + \infty$ for every $\alpha \in A$. The proof is done by contradiction.

Let $E^n_{\alpha,i}$ be the energy of the configuration localized in $I_{\alpha,i}$. More precisely, assume first $x^n_0 < x^n_n$;  index the part of $(x^n_k)_{k=1}^n$  in $I_{\alpha,i}$ by $(x^n_{k,\alpha,i})_{k=1}^N$ with $N=N^n_{\alpha,i}$; denote by $x^n_{0,\alpha,i}$ the nearest  point strictly smaller than $x^n_{1,\alpha,i}$ and define the partial energy $E^n_{\alpha,i}:= E_\omega(x^n_{0,\alpha,i},\ldots,x^n_{N,\alpha,i})$. If $x^n_n < x^n_0$, the part of $(x^n_k)_{k=0}^{n-1}$ in $I_{\alpha,i}$ is indexed by $(x^n_{k,\alpha,i})_{k=0}^{N-1}$ with $N=N^n_{\alpha,i}$; denote by $x^n_{N,\alpha,i}$  the nearest  point  strictly larger than $x^n_{N-1,\alpha,i}$  and define $E^n_{\alpha,i}$ similarly.

Thanks to the hypothesis $\inf_{x\in\mathbb{R}} E_\omega(x,x) > \bar E$, one can choose $\epsilon > 0$ such that $E_\omega(x,y) \geq \bar E + \epsilon$ as soon as $|y-x| \leq \epsilon$. Let $\bar H := \max_{\alpha \in A} H_\alpha$. Then, if $ \theta_{\alpha, i}^n $ denotes the number of consecutive points $x^n_{k,\alpha,i}$ in $I_{\alpha,i}$ satisfying $|x^n_{k,\alpha,i} - x^n_{k-1,\alpha,i}| > \epsilon$, obviously $ \theta_{\alpha, i}^n \le \bar H/\epsilon$.
Thus, since $n =  \sum_{\alpha \in A} \ \sum_{1 \leq i \leq C^n_{\alpha}} N^n_{\alpha,i}$, we have that
\begin{align}
n\bar E &\geq E_\omega(x^n_0,\ldots,x^n_n) = \sum_{\alpha \in A} \ \sum_{1 \leq i \leq C^n_{\alpha}} E^n_{\alpha,i}  \notag \\
&\geq \sum_{\alpha \in A} \  \sum_{1 \leq i \leq C^n_{\alpha}}
\Big[\theta_{\alpha, i}^n \inf_{x,y \in \mathbb R} E_\omega(x,y) + \big (N^n_{\alpha,i} - \theta_{\alpha, i}^n \big) (\bar E+\epsilon)\Big] \notag \\
& = n(\bar E+\epsilon)+  \sum_{\alpha \in A} \sum_{1 \leq i \leq C^n_{\alpha}} \theta_{\alpha, i}^n \underline{E}
\, \ge \,  n(\bar E+\epsilon)+  \sum_{\alpha \in A}C^n_\alpha \frac{\bar H}{\epsilon} \underline{E}, \label{eq:RotationNumber_1}
\end{align}
where $\underline{E} := (\inf_{x, y \in \mathbb R} E_\omega(x,y) - \bar E - \epsilon) < 0 $.  For $\alpha$ fixed, among the intervals $(I_{\alpha,i})_i$,   $i=1,\ldots,C^n_\alpha$, at most two of them  are exceptional and the other intervals satisfy  
$N^n_{\alpha,i} \geq N^n_\alpha - 2$. We thus get $n \geq \sum_{\alpha \in A} (C^n_\alpha-2)(N^n_\alpha-2)$. For $n$ sufficiently large, we have
\begin{gather*}
\frac{C^n_\alpha}{T^n-S^n} \leq (1+\epsilon)\nu_\alpha, \quad  \frac{C^n_\alpha-2}{T^n-S^n} \geq (1-\epsilon) \nu_\alpha \quad\text{and} \\
\frac{1}{n} \sum_{\alpha \in A} C^n_\alpha \leq \frac{(1+\epsilon)\sum_{\alpha \in A} \nu_\alpha}{(1-\epsilon)\sum_{\alpha \in A} \nu_\alpha (N^n_\alpha-2)}.
\end{gather*}
If $N^n_\alpha \to +\infty$ for some $\alpha$ and a subsequence $n \to +\infty$, then $\frac{1}{n} \sum_{\alpha \in A} C^n_\alpha \to 0$ and we obtain a contradiction with the previous inequality \eqref{eq:RotationNumber_1}.

{\it Step 4.} For every $\alpha$,  $I_{\alpha, i} \subset [x^n_0,x^n_n]$ except maybe for at most two of them. Then
\[
\frac{|x^n_n-x^n_0|}{n} \geq \frac{\sum_{\alpha \in A} (C^n_\alpha-2)H_\alpha}{\sum_{\alpha \in A} C^n_\alpha N^n_\alpha}.
\]
Denote $\bar N_\alpha := \limsup_{n\to+\infty}N^n_\alpha$. From step 3 we know that $\bar N_\alpha < +\infty$. By dividing by $(T^n-S^n)$ and by letting $n\to+\infty$, we obtain
\[
\liminf_{n\to+\infty} \frac{|x^n_n-x^n_0|}{n} \geq \frac{\sum_{\alpha \in A} \nu_\alpha H_\alpha}{\sum_{\alpha \in A} \nu_\alpha \bar N_\alpha} = \frac{1}{\sum_{\alpha \in A} \nu_\alpha \bar N_\alpha} >0.
\]
\end{proof}

Now we are able to prove theorem~\ref{theorem:MainContribution}.

\begin{proof}[Proof of theorem~~\ref{theorem:MainContribution}]
Let $ (\Omega,\{\tau_t\}_{t\in\mathbb{R}},L) $ be an almost crystalline interaction model.
We discuss two cases.

{\it Case 1.} Either $\inf_{\omega \in \Omega} \inf_{x \in \mathbb{R}} E_\omega(x,x)=\bar E$. Then $E_{\omega_*}(x_*,x_*)=\bar E$ for some $\omega_*$ and $x_*$. By hypothesis, $L$ is transversally constant with respect to a flow box decomposition $\{U_i =\tau[B_{R_i} \times \Xi_i] \}_{i\in I}$. Let $i\in I$ be such that $\tau_{x_*}(\omega_*) \in U_i$. Let $|t_i| < R_i$ and $\omega_i \in \Xi_i$ be such that 
$\tau_{x_*}(\omega_*) = \tau_{t_i}(\omega_i)$. Then
\[
\bar E = E_{\omega_*}(x_*,x_*) = E_{\omega_i}(t_i,t_i) = E_\omega(t_i,t_i), \quad \forall \omega \in \Xi_i.
\]
We have just proved that $\delta_{(\tau_{t_i}(\omega),0)}$ is a minimizing measure for every $\omega \in \Xi_i$. The projected Mather set contains $\tau_{t_i}(\Xi_i)$. By minimality of the flow, we have $\Omega=\tau[B_R \times \Xi_i]$, for some $R>0$,  thanks to item 1 of lemma \ref{lemma:FlowBoxes}. The projected Mather set thus meets  every sufficiently long orbit of the flow.

{\it Case 2.}  Or $\inf_{\omega \in \Omega} \inf_{x \in \mathbb{R}} E_\omega(x,x)>\bar E$. Proposition \ref{proposition:RotationNumber} shows that,  if  $\omega_* \in \Omega$ has been fixed, if  for every $n\geq 1$ a sequence  $(x_k^n)_{0\leq k < n}$  of points of $\mathbb{R}$ realizing the minimum
$ E_{\omega_*}(x_0^n,\ldots,x_n^n) = \min_{x_0,\ldots,x_n} E_{\omega_*}(x_0,\ldots,x_n) $
has been fixed, then

\noindent -- $\bar E = \lim_{n\to+\infty} \frac{1}{n} E_{\omega_*}(x_0^n,\ldots,x_n^n)$,

\noindent -- $(x_k^n)_{0\leq k < n}$ is strictly monotone for $n$ large enough,

\noindent -- there is $R>0$ (independent of $ \omega_* $) such that $ \sup_{n\geq1} \ \sup_{1 \leq k \leq n}  |x_{k}^n - x_{k-1}^n|  < 2 R$,

\noindent -- $ \rho := \liminf_{n\to+\infty} \frac{1}{n} |x^n_n - x^n_0| > 0$.

Let $\mu_{n,\omega_*}$ be the probability measure on $\Omega\times\mathbb{R}$ defined by
\[
\mu_{n,\omega_*} := \frac{1}{n} \sum_{k=0}^{n-1} \delta_{(\tau_{x_k^n}(\omega_*), \, x_{k+1}^n - \, x_k^n)}.
\]
Notice that $\int\! L\,d\mu_{n,\omega_*} = \frac{1}{n}E_{\omega_*}(x_0^n,\ldots,x_n^n)$. Since the consecutive jumps of $x_k^n$ are uniformly bounded, the sequence of measures $\{\mu_{n,\omega_*}\}_{n\geq1}$ is tight. By taking a subsequence, we may assume that $\mu_{n,\omega_*} \to \mu_\infty$ with respect to the weak topology. Moreover, $\mu_\infty$ is holonomic and minimizing.  Let $\Xi \subset \Omega$ be a transverse section of a flow box $\tau[B_{R}\times \Xi]$.
Let $\mathcal{R}_\Xi(\omega_*)$  be the set of return times to $ \Xi $ as defined in lemma~\ref{lemma:TransverseMeasure}.
Let $pr^1:\Omega\times\mathbb{R} \to \Omega$ be the first projection. Then
\begin{align*}
pr^1_*(\mu_{n,\omega_*})(\tau[ B_R\times \Xi]) &= \frac{1}{n} \# \big\{k : x_k^n \in \cup_{t \in \mathcal{R}_{\Xi}(\omega_*)}  B_R(t) \big\} \\
&\geq \frac{1}{n} \#(B_{T_n}(c_n) \cap \mathcal{R}_\Xi(\omega_*)),
\end{align*}
with $T_n := \frac{1}{2} |x^n_n-x^n_0|$ and $c_n := \frac{1}{2}(x^n_0+x_n^n)$. The previous inequality comes from the fact that the intervals $ B_R(t)$ are disjoints and contain at least one $x_k^n$. Then
\begin{align*}
pr^1_*(\mu_{n,\omega_*})(\tau[ B_R\times \Xi]) &\geq \frac{2T_n}{n} \ \frac{\#(B_{T_n}(0) \cap \mathcal{R}_\Xi(\tau_{c_n}(\omega_*))}{\text{\rm Leb}(B_{T_n}(0))}.
\end{align*}
By taking the limit as $n\to+\infty$, one obtains $pr^1_*(\mu_\infty)(\overline{\tau[ B_R\times\Xi]}) \geq \rho \nu_\Xi(\Xi)>0$.
Therefore, since $ \Xi $ is arbitrary, every orbit of the flow of length $2R$ meets the projected Mather set.
\end{proof}

\end{document}